\documentclass{amsart}

\usepackage{graphicx}
\usepackage{mathtools}
\usepackage{hyperref}

\newtheorem{thm}{Theorem}[section]
\newtheorem{lem}[thm]{Lemma}
\newtheorem{cor}[thm]{Corollary}
\newtheorem{prop}[thm]{Proposition}

\theoremstyle{definition}

\theoremstyle{remark}
\newtheorem{rem}[thm]{Remark}

\numberwithin{equation}{section}
\newcommand{\Thm}[1]{Theorem~\ref{th:#1}}
\newcommand{\Prop}[1]{Proposition~\ref{prop:#1}}
\newcommand{\Lem}[1]{Lemma~\ref{lem:#1}}
\newcommand{\Cor}[1]{Corollary~\ref{cor:#1}}

\newcommand{\Eq}[1]{\eqref{eq:#1}}
\newcommand{\Sec}[1]{Section~\ref{sec:#1}}
\newcommand{\Fig}[1]{Figure~\ref{fig:#1}}
\makeatletter
\def\rom#1{\mbox{\leavevmode\skip@\lastskip\unskip\/\ifdim\skip@=\z@\else\hskip\skip@\fi{\rm{#1}}}}
\makeatother

\renewcommand{\a}{\alpha}\renewcommand{\b}{\beta}
\newcommand{\gm}{\gamma}\newcommand{\dl}{\delta}
\newcommand{\eps}{\varepsilon}

\newcommand{\lm}{\lambda}\newcommand{\kp}{\kappa}
\newcommand{\sg}{\sigma}
\newcommand{\ph}{\varphi}
\newcommand{\om}{\omega}

\newcommand{\E}{\mathbb{E}}
\newcommand{\N}{\mathbb{N}}
\newcommand{\R}{\mathbb{R}}
\newcommand{\Z}{\mathbb{Z}}
\newcommand{\cB}{\mathcal{B}}
\newcommand{\cC}{\mathcal{C}}
\newcommand{\cE}{\mathcal{E}}
\newcommand{\cF}{\mathcal{F}}
\newcommand{\cH}{\mathcal{H}}
\newcommand{\dd}{\mathrm{d}}

\newcommand{\bfone}{\mathbf{1}}

\DeclareRobustCommand{\nuesssup}{\mathop{\nu\text{-}\mathrm{ess\,sup}}}
\DeclareRobustCommand{\nuessinf}{\mathop{\nu\text{-}\mathrm{ess\,inf}}}
\DeclareMathOperator{\diam}{diam}
\DeclareMathOperator{\tr}{tr}
\allowdisplaybreaks[3]
\begin{document}

\title[Energy densities on $N$-dimensional Sierpinski gaskets]{Pointwise regularity and irregularity of energy densities on $N$-dimensional Sierpinski gaskets}


\author[M. Hino]{Masanori Hino}
\address{Department of Mathematics, Graduate School of Science, Kyoto University, Kyoto 606-8502, Japan}
\email{hino@math.kyoto-u.ac.jp}
\thanks{This work was supported by JSPS KAKENHI Grant Number 25K07056 and The Kyoto University Foundation.}

\author[K. Inui]{Kanji Inui}
\address{Department of Mechanical Engineering and Science, Faculty of Science and Engineering, Doshisha University, Kyoto 610-0394, Japan}
\email{kinui@mail.doshisha.ac.jp}
\author[K. Nitta]{Kohei Nitta}
\address{Nishinomiya-shi, Hyogo, Japan}
\email{knitta1127@gmail.com}


\date{}

\dedicatory{}

\begin{abstract}
We study the pointwise regularity of energy densities associated with harmonic functions on the $N$-dimensional Sierpinski gasket $(N\ge 2)$ with respect to the Kusuoka measure. For any nonconstant harmonic function, we prove that every Borel representative of the density is discontinuous at every point of a set of full Kusuoka measure. In sharp contrast, on each one-dimensional edge of the gasket---itself a set of zero Kusuoka measure---the density admits a canonical pointwise version, which is $\gm_N$-H\"older continuous on that edge with the explicit and optimal exponent
$\gm_N= \log_2 \{(\sqrt{4N+5} + 1) / (\sqrt{4N+5}-1)\}$.
\end{abstract}

\maketitle
\section{Introduction}\label{sect:1}

Energy measures are fundamental objects associated with local Dirichlet forms.
As a prototypical example, consider the standard energy form
$(\cE,H^1(\mathbb{R}^d))$ on $\mathbb{R}^d$, defined by
\[
    \cE(f,g)=\frac{1}{2}\int_{\R^d}(\nabla f,\nabla g)\,\dd x,
    \qquad f,g\in H^1(\R^d).
\]
The energy measure $\nu_f$ of $f\in H^1(\R^d)$ is then given by
\[
    \nu_f(\dd x)=|\nabla f(x)|^2\,\dd x.
\]
Thus, in the Euclidean case, the Radon--Nikodym derivative of the energy
measure is nothing but the squared length of the gradient. This simple
description relies essentially on the existence of a classical gradient
operator.

\begin{figure}[t]
\centering
\includegraphics[width=0.36\textwidth]{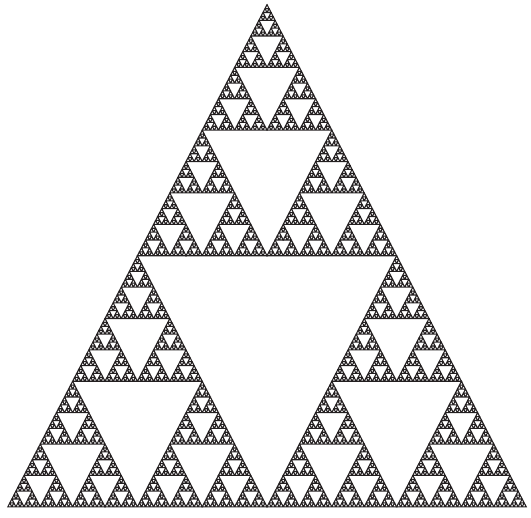}
\caption{The two-dimensional Sierpinski gasket.}
\label{fig:SG}
\end{figure}

For canonical Dirichlet forms on typical fractals, the situation is quite
different. These forms are constructed by renormalized discrete approximations
and do not possess a classical gradient operator. Accordingly, their energy
measures have no simple local expression, and even their basic quantitative
properties are highly nontrivial. For a class of self-similar fractals including
Sierpinski gaskets (see \Fig{SG}), Kusuoka~\cite{Ku89} represented energy measures in terms of
infinite products of matrices. He also proved that these measures are singular
with respect to the canonical self-similar measure, while they are absolutely
continuous with respect to the so-called Kusuoka measure.

The Kusuoka measure therefore provides a natural reference measure for
energy measures on Sierpinski gaskets.  However, even after this reference
measure has been chosen, the Radon--Nikodym derivative of the energy
measure of a function $h$ is defined only $\nu$-a.e.  It is
therefore natural to ask whether this density admits any canonical
pointwise meaning on geometrically distinguished subsets of the gasket,
in particular on one-dimensional edges. 
The main point of the present paper is to show that, despite the
measure-theoretic irregularity of the density, such a canonical edge trace
does exist and has sharp H\"older regularity.

This phenomenon was first observed by Bell, Ho, and Strichartz~\cite{BHS14} for the
two-dimensional Sierpinski gasket. Let $K$ be the two-dimensional Sierpinski
gasket. Let $V_*$ be the set of all vertices and junction points
of $K$, which is a countable dense subset of $K$. For a harmonic function
$h$ on $K$, let $\nu_h$ denote its energy measure, and let $\nu$ denote
the Kusuoka measure. Bell, Ho, and Strichartz proved the following remarkable
properties.
\begin{enumerate}
\item
For any nonconstant harmonic function $h$, any Borel $\nu$-version of the Radon--Nikodym derivative $\dd\nu_h/\dd\nu$ is discontinuous at every point of $K$.
\item
There exists a natural version $\delta\nu_h/\delta\nu$ of $\dd\nu_h/\dd\nu$ on
the intersection of $V_*$ with each edge $E$ of $K$, and this version is
continuous at each point along $E\cap V_*$.
\end{enumerate}
\begin{figure}[!t]
\centering
\includegraphics[width=0.38\textwidth]{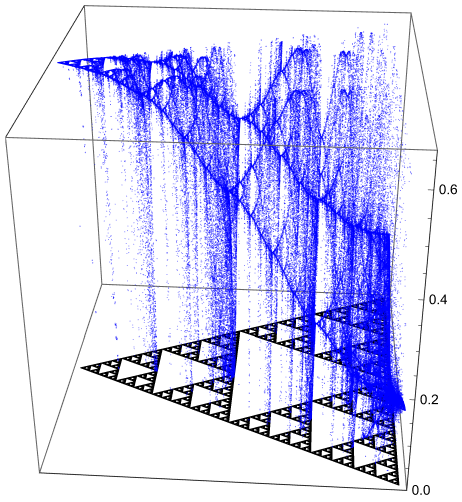}\quad
\includegraphics[width=0.38\textwidth]{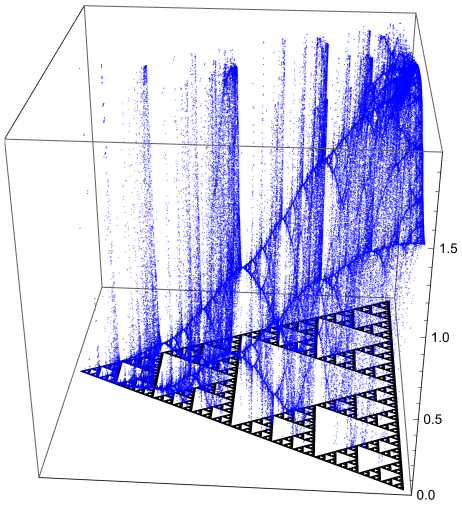}
\caption{Level-11 cell-average approximations of the Radon--Nikodym derivatives $\dd\nu_h/\dd\nu$ for harmonic functions $h$. Left figure: For $h$ with values $1,0,0$ at the vertices. Right figure: For $h$ with values $0,1,-1$ at the vertices.}
\label{fig:density}
\end{figure}
\Fig{density} illustrates these two features numerically. Since each edge of the
gasket is a $\nu$-null set (see \Lem{null} below), the existence of such a natural version on edges is
not a consequence of the usual Radon--Nikodym theorem. It is a genuinely
pointwise statement about energy densities.

The purpose of the present paper is to extend and refine this picture for the
$N$-dimensional Sierpinski gasket, where $N\ge 2$. Our first main result is
that the measure-theoretic irregularity persists in all dimensions. More
precisely, if $h$ is a nonconstant harmonic function, then there exists a
Borel set $\hat K\subset K$ with $\nu(K\setminus \hat K)=0$ such that every Borel
$\nu$-version of $\dd\nu_h/\dd\nu$ is discontinuous at every point of $\hat K$
(\Thm{main1}). Thus no choice of representative can make the density regular on
a set of full Kusuoka measure. Compared with the two-dimensional result of
\cite{BHS14}, our statement contains a possible exceptional $\nu$-null set, but the
proof works uniformly for all $N\ge 2$ and, apart from the positivity of $\dd\nu_h/\dd\nu$ established in \cite[Theorem 5.6]{Hi10}, uses only basic features of the Dirichlet form and harmonic functions.

Our second main result identifies a canonical edge trace of the energy
density.  More precisely, for each point on an edge, we consider the ratios
of the energy measure $\nu_h$ to the Kusuoka measure $\nu$ on the nested
sequence of cells shrinking to that point along the edge.  We prove that
these ratios converge and that the limit is independent of the symbolic
representation of the point.  We call the resulting function on the edge
the canonical edge trace of $\dd\nu_h/\dd\nu$.  This trace is
$\gamma_N$-H\"older continuous on every edge, where
\[
    \gamma_N=\log_2\frac{\sqrt{4N+5}+1}{\sqrt{4N+5}-1}.
\]
Moreover, this exponent is optimal in the following sense: for any
nonconstant harmonic function, there is at least one edge on which the
H\"older exponent of its canonical edge trace cannot be improved; conversely,
on any prescribed edge, such an example can be found
(Theorems~\ref{th:main3} and \ref{th:main4}).
We emphasize that this result is stronger than the
previously known continuity statement in the two-dimensional case.  The
canonical edge trace is defined on the whole edge, not only on junction
points, and its optimal H\"older exponent is explicitly determined.  To the
best of our knowledge, such a sharp trace regularity result for energy
densities on fractals has not been observed before.

The proofs are not merely higher-dimensional adaptations of the arguments in \cite{BHS14}. The proof of the almost-everywhere discontinuity result is based on a detailed study of the behavior of
harmonic functions along edges. In particular, we analyze the location of
extremal points of harmonic functions restricted to an edge, a result which is
of independent interest. 
Although the outline of the proof is similar to that in \cite{BHS14}, we use some techniques different from \cite{BHS14} to handle higher-dimensional Sierpinski gaskets.
The arguments on the regularity properties are entirely different from those in~\cite{BHS14}.
The proof of the H\"older regularity theorem is based
on Hilbert's projective metric and on the contraction properties of certain
iterated maps on invariant cones. The optimality of the H\"older
exponent is proved by an explicit analysis at suitable periodic points.

The Radon--Nikodym derivative $\dd\nu_h/\dd\nu$ is closely related to the abstract
gradient of $h$ in the framework of measurable Riemannian structures
associated with Dirichlet forms~\cite{Hi13}. In that general framework, gradients
are defined only up to $\nu$-null sets. The existence of a canonical edge trace on edges, which are themselves $\nu$-null sets,
suggests that pointwise first-order structures on fractals may be richer than
what is visible from the measure-theoretic formulation alone.
In this sense, the edge trace is not merely a refinement of an a.e.\ representative, but rather an additional pointwise object selected by the cell structure.

The remainder of the paper is organized as follows. In \Sec{statements}, we introduce the
framework and state the main theorems precisely. In \Sec{pre}, we collect
preliminary facts on harmonic functions, energy measures, and the associated
matrices. In \Sec{location}, we study the behavior of harmonic functions on edges,
including the location of their extremal points. \Sec{main1} is devoted to the
proof of \Thm{main1}. In \Sec{main3}, we prove the existence and H\"older
continuity of the canonical edge trace of the energy density. Finally, in
\Sec{main4}, we prove the optimality of the H\"older exponent.

\section{Framework and statement of main theorems}\label{sec:statements}
We first set up the framework. See, e.g., \cite{Ku89,Ku93,Hi10} for more details.
Let $N$ be an integer with $N\ge2$. We construct the $N$-dimensional Sierpinski gasket as follows. Consider a closed regular $N$-simplex $\tilde K$ in $\R^N$ and let $V_0=\{p_1,p_2,\dots,p_{N+1}\}$ be the set of its vertices. In other words, the convex hull of $V_0$ is equal to $\tilde K$.
For each $i=1,2,\dots, N+1$, a contraction map $\psi_i\colon \R^{N}\to\R^{N}$ is defined as $\psi_i(z)=(z+p_i)/2$.
Let $S=\{1,2,\dots,N+1\}$ and $W_m$ denote $\underbrace{S\times\cdots\times S}_m$ for $m\in \N$. We set $W_0=\{\emptyset\}$ by convention, and set $W_*=\bigcup_{m=0}^\infty W_m$. An element of $W_*$ is called a \emph{word}.
For $w\in W_m$, we call $m$ the \emph{length} of $w$ and write it as $|w|$.
For $i\in S$ and $n\in\N$, $i^n\in W_n$ denotes $\underbrace{ii\cdots i}_n$.
For $w=w_1w_2\cdots w_m\in W_m$ and $w'=w'_1w'_2\cdots w'_{m'}\in W_{m'}$, $\psi_w$ denotes the contraction map $\psi_{w_1}\circ\psi_{w_2}\circ\cdots\circ\psi_{w_m}$ on $\R^N$, and $ww'=w_1w_2\cdots w_m w'_1w'_2\cdots w'_{m'}\in W_{m+m'}$. Here, $\psi_\emptyset$ is defined as the identity map by convention.
For $\om=\om_1\om_2\cdots\in S^\N$ and $n\in\N$, $[\om]_n$ denotes $\om_1\om_2\cdots\om_n\in W_n$. Also, $i^\infty$ means $iii\cdots\in S^\N$ for $i\in S$. Similarly, $(ij)^\infty$ denotes $ijijij\cdots\in S^\N$ for $i,j\in S$.
For $w\in W_*$ and $\om\in S^\N$, $w\om\in S^\N$ is defined in the natural way.

The $N$-dimensional Sierpinski gasket $K$ is defined as $K=\bigcap_{m=0}^\infty \bigcup_{w\in W_m}\psi_w(\tilde K)$. This is the unique nonempty compact subset of $\R^N$ such that the relation $K=\bigcup_{i\in S}\psi_i(K)$ holds.
The metric on $K$ is induced by the Euclidean metric on $\R^N$.
For $w\in W_*$, $K_w$ is defined as $K_w=\psi_w(K)$.
Let $V_m=\bigcup_{w\in W_m}\psi_w(V_0)$ for $m\in\N$ and define $V_*=\bigcup_{m=0}^\infty V_m$. $V_*$ is a dense subset of $K$.

For a general set $V$, $l(V)$ denotes the space of all real functions on $V$. When $V$ is a finite set, the standard inner product on $l(V)$ is denoted by $(\cdot,\cdot)$, that is, $(f,g)=\sum_{x\in V}f(x)g(x)$.
The associated norm is denoted by $|\cdot|$.
We identify $l(V_0)$ with $\R^S\simeq\R^{N+1}$.
Let $D$ be a linear operator on $l(V_0)$ defined as $D_{p,q}=\begin{cases}-N &(p=q)\\1&(p\ne q)\end{cases}$, that is, in the matrix form,
\begin{equation}\label{eq:D}
D=\begin{pmatrix}
-N&1&\cdots&\cdots&1\\
1&-N&\ddots&&\vdots\\
\vdots&\ddots&\ddots&\ddots&\vdots\\
\vdots&&\ddots&-N&1\\
1&\cdots&\cdots&1&-N
\end{pmatrix}.
\end{equation}
The kernel of $D$ is the set of all constant functions.
A bilinear form $Q_0$ on $l(V_0)$ is defined as
\[
Q_0(f,g):=\frac12\sum_{i,j\in S}(f(p_i)-f(p_j))(g(p_i)-g(p_j))
=(f,-D g),\qquad
f,g\in l(V_0).
\]
We define a bilinear form $Q_m$ on $l(V_m)$ for $m\in\N$ as
\[
Q_m(f,g)=\sum_{w\in W_m}Q_0(f\circ \psi_w|_{V_0},g\circ \psi_w|_{V_0}),\qquad
f,g\in l(V_m).
\]
Then, for each $f\in l(V_*)$, one checks that $\left\{\left(\frac{N+3}{N+1}\right)^m Q_m(f|_{V_m},f|_{V_m})\right\}_{m=0}^\infty$ is a nondecreasing sequence. Its limit is denoted by $\cE_*(f,f)\in[0,\infty]$.
Let $\cF_*=\{f\in l(V_*)\mid \cE_*(f,f)<\infty\}$.
Each element $f$ in $\cF_*$ is shown to be uniformly continuous on $V_*$, so $f$ can extend to a continuous function on $K$.
We now define
\begin{align*}
\cF&=\{f\in C(K)\mid f|_{V_*}\in \cF_*\},\\
\cE(f,g)&=\frac12\{\cE_*((f+g)|_{V_*},(f+g)|_{V_*})-\cE_*(f|_{V_*},f|_{V_*})-\cE_*(g|_{V_*},g|_{V_*})\},\quad f,g\in\cF.
\end{align*}
Here, $C(K)$ denotes the set of all real continuous functions on $K$.
Let $\mu$ be the normalized Hausdorff measure on $K$.
By identifying $C(K)$ with a subset of $L^2(K,\mu)$, $(\cE,\cF)$ is a strongly local regular Dirichlet form on $L^2(K,\mu)$.
We note the following self-similar property holds: $\psi_w^* f\in\cF$ for each $f\in\cF$ and $w\in W_*$, and
\begin{equation}\label{eq:selfsimilarity}
\cE(f,g)=\left(\frac{N+3}{N+1}\right)^m\sum_{w\in W_m}\cE(\psi_w^* f,\psi_w^* g),\qquad f,g\in\cF
\end{equation}
for each $m\in\N$.
Here, $\psi_w^* f$ denotes the pullback $f\circ\psi_w$ of $f$ by $\psi_w$.
The energy measure $\nu_f$ of $f\in\cF$ is defined\footnote{This definition is exactly twice the one used in \cite{BHS14}, but this difference does not affect the subsequent arguments.} as the unique positive Borel measure on $K$ such that
\begin{equation}\label{eq:energy_measure}
\int_K g\,\dd\nu_f=2\cE(f,fg)-\cE(f^2,g)\quad
\text{for every }g\in\cF.
\end{equation}
In particular, $\nu_f(K)=2\cE(f,f)$.
From \Eq{selfsimilarity},
\begin{equation}\label{eq:energyselfsimilarity}
\int_K g\,\dd\nu_f=\left(\frac{N+3}{N+1}\right)^m\sum_{w\in W_m}\int_K \psi_w^*g\,\dd\nu_{\psi_w^* f}
\end{equation}
for every $m\ge0$ and $g\in\cF$. The monotone class theorem implies that \Eq{energyselfsimilarity} holds for any bounded Borel functions $g$.
From the general theorem of strongly local regular Dirichlet forms, $\nu_f$ does not charge any single points (see, e.g., \cite[Theorem~4.3.8]{CF12}). From this property and \Eq{energyselfsimilarity}, it holds that
\begin{equation}\label{eq:cellenergy}
\nu_f(K_w)=2\left(\frac{N+3}{N+1}\right)^{|w|}\cE(\psi_w^* f,\psi_w^* f),\quad w\in W_*.
\end{equation}
For each $u\in l(V_0)$, there exists a unique function $h\in\cF$ such that $h|_{V_0}=u$ and
\[
\cE(h,h)=\inf\{\cE(f,f)\mid f\in\cF\text{ and }f|_{V_0}=u\}.
\]
We call such a function $h$ \emph{harmonic}. The space of harmonic functions is denoted by $\cH$.
We define the map $\iota\colon l(V_0)\to \cH$ by $\iota(u)=h$ in the above correspondence. 
Note that $\psi_w^* h\in\cH$ for any $h\in\cH$ and $w\in W_*$.
For each $k\in S$, let $e_k\in l(V_0)$ be defined as $e_k(p_i)=\begin{cases}1&(i=k)\\0&(i\ne k)\end{cases}$. In the natural manner, $e_1,e_2,\dots,e_{N+1}$ are identified with the standard basis of $\R^{N+1}$.
Let $h_k\in\cH$ denote $\iota(e_k)$. 
The Kusuoka measure $\nu$ is defined as $\nu=\sum_{k\in S}\nu_{h_k}$.
Then, for every $f\in\cF$, $\nu_f$ is absolutely continuous with respect to $\nu$, which follows from \cite[Lemma~5.1]{Ku89} (see also \cite[Lemma~5.7]{Hi10}).

The first main result in this paper is stated as follows.
\begin{thm}\label{th:main1}
Let $h\in \cH$ be a nonconstant function. Then, there exists a Borel subset $\hat K$ of $K$ such that $\nu(K\setminus \hat K)=0$ and every Borel $\nu$-version of the Radon--Nikodym derivative $\dd\nu_h/\dd \nu$ is discontinuous at every point of $\hat K$.
\end{thm}

For $p,p'\in K$, $\overline{pp'}$ denotes the line segment connecting $p$ and $p'$.
Let $w\in W_*$ and $i,j\in S$ be distinct elements.
For $\om\in\{i,j\}^\N$, $\bigcap_{n=0}^\infty K_{w[\om]_n}=\{x\}$ for some $x\in\psi_w(\overline{p_ip_j})$. We describe this correspondence as $\Psi_w(\om)=x$.
The map $\Psi_w$ is surjective but not injective; for instance, $\Psi_w(ij^\infty) = \Psi_w(ji^\infty)$.
By an edge we mean a line segment of the form $\psi_w(\overline{p_ip_j})$, where $w\in W_*$ and $i\ne j$.
The construction and H\"older regularity of the canonical edge trace are
described as follows.
\begin{thm}\label{th:main3}
Let $w\in W_*$ and $i,j\in S$ be distinct elements. Let $h\in\cH$ and $\om\in\{i,j\}^\N$.
Then,
\begin{equation}\label{eq:dlh}
\dl_{h,w}(\om):=\lim_{n\to\infty}\frac{\nu_h(K_{w [\om]_n})}{\nu(K_{w [\om]_n})}
\end{equation}
exists. The function $\frac{\dl \nu_h}{\dl \nu}(x):=\dl_{h,w}(\om)$ for $x=\Psi_w(\om)$, $x\in\psi_w(\overline{p_ip_j})$ is well-defined and $\gm_N$-H\"older continuous on $\psi_w(\overline{p_ip_j})$ with $\gm_N=\log_2\frac{\sqrt{4N+5}+1}{\sqrt{4N+5}-1}$.
\end{thm}
We refer to $\frac{\dl \nu_h}{\dl \nu}$ as the canonical edge trace of $\frac{\dd \nu_h}{\dd \nu}$.
\begin{rem}
Consider the filtration $\{\cF_n\}_{n=0}^\infty$ of the Borel $\sg$-field of $K$ that is defined as $\cF_n=\sg(\{K_w\mid w\in W_n\})$, $n\in\Z_+$.
By applying the martingale convergence theorem to the process $\left\{\E\left[\frac{\dd\nu_h}{\dd\nu}\mathrel{}\middle|\mathrel{}\cF_n\right]\right\}_{n=0}^\infty$, where $\E[\,\cdot\mid\cF_n]$ denotes the conditional expectation given $\cF_n$ with respect to the probability measure obtained by normalizing $\nu$, we find that for $\nu$-a.e.\,$x\in K\setminus V_*$, $\nu_h(K_{[\kp]_n})/\nu(K_{[\kp]_n})$ converges to $\frac{\dd\nu_h}{\dd\nu}(x)$ as $n\to\infty$, where $\kp$ is the unique element in $S^\N$ such that $\bigcap_{n=0}^\infty K_{[\kp]_n}=\{x\}$.
Note that for each $w\in W_n$ the conditional expectation $\E\left[\frac{\dd\nu_h}{\dd\nu}\mathrel{}\middle|\mathrel{}\cF_n\right]$ equals $\nu_h(K_w)/\nu(K_w)$ $\nu$-a.e.\ on $K_w$, because distinct cells of level $n$ intersect only in $V_n$ and $\nu(V_n)=0$.
This justifies that $\frac{\dl \nu_h}{\dl \nu}$ is a natural representative of $\frac{\dd\nu_h}{\dd\nu}$, although every edge of $K$ is a $\nu$-null set, as shown in \Lem{null} below.
\end{rem}
We note that $\gm_N$ is strictly decreasing in $N$, and $0<\gm_N\le \gm_2=\log_2\frac{\sqrt{13}+1}{\sqrt{13}-1}\approx 0.821785$.
The optimality of the H\"older exponent $\gm_N$ is described as follows.
\begin{thm}\label{th:main4}
Let $w\in W_*$.
\begin{enumerate}
\item For any nonconstant $h\in\cH$, there exist distinct $i,j\in S$ and $x\in\psi_w(\overline{p_ip_j})$ such that
\begin{equation}\label{eq:reverse}
\limsup_{y\in\psi_w(\overline{p_ip_j})\setminus\{x\},\ y\to x}\frac{\left|\frac{\dl \nu_h}{\dl \nu}(y)-\frac{\dl \nu_h}{\dl \nu}(x)\right|}{|y-x|^{\gm_N}}>0.
\end{equation}
\item For any distinct $i,j\in S$, there exist $h\in\cH$ and $x\in\psi_w(\overline{p_ip_j})$ such that \Eq{reverse} holds.
\end{enumerate}
\end{thm}
\begin{rem}
As seen from the proof of \Thm{main4}~(ii), in that case we may take
$x$ to be the point dividing the line segment $\psi_w(\overline{p_ip_j})$
internally in the ratio $1:2$.  
In \Thm{main4}~(i), the proof implies that we may take  $x=\Psi_w(i^m (ij)^\infty)$ for some $m\in\{0,1,2\}$.  Equivalently, $x$ is the point dividing
the line segment $\psi_{wi^m}(\overline{p_ip_j})$ internally in the ratio $1:2$.
\end{rem}

\section{Preliminaries}\label{sec:pre}
For each $k\in S$, the representation matrix corresponding to the linear map $l(V_0)\ni u\mapsto \iota(u)\circ\psi_k|_{V_0}\in l(V_0)$ is denoted by $A_k$.
The explicit expression of $A_k=(a_{ij}^{(k)})_{i,j=1}^{N+1}$ is given by
\[
a_{ij}^{(k)}=
\begin{cases}
1 & (i=j=k), \\
0 & (i=k,\; j\neq k), \\
\dfrac{2}{N+3} & (i\neq k,\; j=k),\smallskip\\
\dfrac{2}{N+3} & (i=j,\; i\neq k),\smallskip\\
\dfrac{1}{N+3} & (\text{otherwise})
\end{cases}
\]
(see, e.g., \cite{Ku89}). In particular, $A_1$ is expressed as
\begin{equation}\label{eq:A1}
A_1=\frac{1}{N+3}
\begin{pmatrix}
N+3 & 0 & \cdots & \cdots & 0 \\
2   & 2 & 1      & \cdots & 1 \\
\vdots & 1 & 2 & \ddots & \vdots \\
\vdots & \vdots & \ddots & \ddots & 1 \\
2 & 1 & \cdots & 1 & 2
\end{pmatrix}.
\end{equation}
The matrix $A_k$ for $k\in S\setminus\{1\}$ is obtained from $A_1$ by simultaneously permuting the first and the $k$th rows and columns.

Let $\bfone\in l(V_0)\simeq \R^{N+1}$ denote the constant function on $V_0$ taking value $1$.
For $k\in S$, we define an $(N-1)$-dimensional subspace $E_k$ of $l(V_0)$ as
\begin{equation}\label{eq:Ek}
E_k=\{ u\in l(V_0)\mid (u,e_k)=0\text{ and }(u,\bfone)=0\}.
\end{equation}
Let $d_k$ denote the $k$th column of the matrix $D$ (see \Eq{D}).
\begin{lem}\label{lem:eigen}
Let $k\in S$.
All the eigenvalues of $A_k$ and $^t\! A_k$ are $1$, $(N+1)/(N+3)$, and $1/(N+3)$. The corresponding eigenspaces of $A_k$ are $\R\bfone$, $\R(\bfone-e_k)$, and $E_k$, respectively. The corresponding eigenspaces of $^t\! A_k$ are $\R e_k$, $\R d_k$, and $E_k$, respectively.
\end{lem}
\begin{proof}
First, we discuss $A_k$. It is easy to see that $A_k\bfone=\bfone$ and $A_k(\bfone-e_k)=\frac{N+1}{N+3}(\bfone-e_k)$.
Also, for $u={}^t(u_1,u_2,\dots,u_{N+1})\in\R^{N+1}$,
\begin{align*}
A_k u=\frac1{N+3}u
&\iff u_k=0\text{ and }\sum_{j\in S\setminus\{k\}}u_j=0\\
&\iff u\in E_k.
\end{align*}
Since the linear span of $\R\bfone$, $\R(\bfone-e_k)$, and $E_k$ is $(N+1)$-dimensional, it coincides with $\R^{N+1}$. In other words, there are no other eigenvalues and eigenvectors.

The proof of the claims for $^t\!A_k$ is similar. It is easy to see that $^t\!A_k e_k=e_k$ and $^t\!A_k d_k=\frac{N+1}{N+3}d_k$.
Also, for $u={}^t(u_1,u_2,\dots,u_{N+1})\in\R^{N+1}$,
\begin{align*}
^t\!A_k u=\frac1{N+3}u
&\iff (N+2)u_k+2\sum_{j\in S\setminus\{k\}}u_j=0\text{ and }\sum_{j\in S\setminus\{k\}}u_j=0\\
&\iff u\in E_k.
\end{align*}
By the same argument as above, there are no other eigenvalues and eigenvectors.
\end{proof}
For $w=w_1w_2\cdots w_m\in W_m$, $A_w$ denotes $A_{w_m}A_{w_{m-1}}\cdots A_{w_1}$.
Let
\[
\tilde l(V_0)=\{u\in l(V_0)\mid (u,\bfone)=0\}
\]
and $P\colon l(V_0)\to\tilde l(V_0)\subset l(V_0)$ denote the orthogonal projection onto $\tilde l(V_0)$. 
For $i\in S$, define
\[
v_i:=\frac1{(d_i,\bfone-e_i)}(\bfone-e_i)=\frac1{N}(\bfone-e_i).
\]
We note that $(d_i,v_i)=1$ and $Q_0(v_i,v_i)=1/N$.
The following is a basic asymptotic behavior.
\begin{lem}[See {\cite[Lemmas~6 and 7]{HN06}}]\label{lem:2}
For $i\in S$ and $u\in l(V_0)$, the following hold.
\begin{align*}
\lim_{n\to\infty}\left(\frac{N+3}{N+1}\right)^n P A_i^n u&=(d_i,u)P v_i,\\
\lim_{n\to\infty}\left(\frac{N+3}{N+1}\right)^n \nu_{\iota(u)}(K_{i^n})&=2(d_i,u)^2 Q_0(v_i,v_i)=\frac2N (d_i,u)^2.
\end{align*}
\end{lem}
From \Eq{cellenergy}, it holds for $f\in\cF$ and $w,w'\in W_*$ that
\begin{equation}\label{eq:ww'}
\nu_f(K_{ww'})=\left(\frac{N+3}{N+1}\right)^{|w|}\nu_{\psi_w^* f}(K_{w'}).
\end{equation}
Indeed, both sides of \Eq{ww'} are equal to $2\left(\frac{N+3}{N+1}\right)^{|w|+|w'|}\cE(\psi_{w'}^*\psi_w^* f,\psi_{w'}^*\psi_w^* f)$.
\begin{lem}\label{lem:3}
For $i\in S$, $w\in W_*$, and $u\in l(V_0)$,
\[
\lim_{n\to\infty}\left(\frac{N+3}{N+1}\right)^n \nu_{\iota(u)}(K_{w i^n})=\frac2N \left(\frac{N+3}{N+1}\right)^{|w|}(d_i,A_w u)^2 .
\]
\end{lem}
\begin{proof}
From \Lem{2} and \Eq{ww'}, we have
\begin{align*}
\left(\frac{N+3}{N+1}\right)^n \nu_{\iota(u)}(K_{w i^n})
&=\left(\frac{N+3}{N+1}\right)^n \left(\frac{N+3}{N+1}\right)^{|w|}\nu_{\psi_w^* (\iota(u))}(K_{i^n})\\
&\xrightarrow{n\to\infty}\frac2N \left(\frac{N+3}{N+1}\right)^{|w|}(d_i,A_w u)^2.\qedhere
\end{align*}
\end{proof}
\begin{prop}\label{prop:4}
For $i\in S$ and $w\in W_*$,
\[
\nu(K_{w i^n})=\Theta\left(\left(\frac{N+1}{N+3}\right)^n\right)
\qquad\text{as }n\to\infty.
\]
Here, $a_n=\Theta(b_n)$ as $n\to\infty$ means $a_n=O(b_n)$ and $b_n=O(a_n)$ as $n\to\infty$.
\end{prop}
\begin{proof}
From \Lem{3},
\[
\lim_{n\to\infty}\left(\frac{N+3}{N+1}\right)^n \nu(K_{w i^n})=\frac2N \left(\frac{N+3}{N+1}\right)^{|w|}\sum_{k\in S}(d_i,A_w e_k)^2.
\]
Here, $\sum_{k\in S}(d_i,A_w e_k)^2=|{}^t\!A_w d_i|^2$, which is strictly positive since $A_w$ is invertible and $d_i$ is a non-zero vector. This implies the claim.
\end{proof}
\begin{lem}\label{lem:simple}
Let $h\in\cH$. Then, $\dd\nu_h/\dd\nu\le |P\iota^{-1}(h)|^2$ $\nu$-a.e.
\end{lem}
\begin{proof}
If $P\iota^{-1}(h)=0$, $\nu_h$ is a zero measure and the claim is true.
Suppose $P\iota^{-1}(h)\ne0$. Let $u_1=P\iota^{-1}(h)/|P\iota^{-1}(h)|$.
Choose $u_2,u_3,\dots,u_{N+1}\in l(V_0)$ so that $u_1,u_2,\dots,u_{N+1}$ form an orthonormal basis of $l(V_0)$. Then, for any $w\in W^*$,
\begin{align*}
\sum_{i=1}^{N+1} \nu_{\iota(u_i)}(K_w)
&=2\left(\frac{N+3}{N+1}\right)^{|w|}\sum_{i=1}^{N+1}(A_w u_i,-D A_w u_i) \qquad\text{(from \Eq{cellenergy})}\\
&=2\left(\frac{N+3}{N+1}\right)^{|w|}\tr(-{{}^t\!A_w}D A_w)\\
&=2\left(\frac{N+3}{N+1}\right)^{|w|}\sum_{i=1}^{N+1}(A_w e_i,-D A_w e_i)\\
&=\nu(K_w).
\end{align*}
That is, $\sum_{i=1}^{N+1} \nu_{\iota(u_i)}=\nu$. Therefore,
\[
\frac{\dd\nu_h}{\dd\nu}
=|P\iota^{-1}(h)|^2\frac{\dd\nu_{\iota(u_1)}}{\dd\nu}
\le |P\iota^{-1}(h)|^2 \quad\text{$\nu$-a.e.}\qedhere
\]
\end{proof}
The following is a well-known fact, but we give a proof for completeness.
\begin{lem}\label{lem:null}
For every $w\in W_*$ and every distinct $i,j\in S$,
$\nu(\psi_w(\overline{p_ip_j}))=0$.
\end{lem}
\begin{proof}
It suffices to prove that $\nu_h(\psi_w(\overline{p_ip_j}))=0$ for $h\in\cH$.  
Take $k\in S\setminus\{i,j\}$.
From the invertibility of $P A_k$ on $\tilde l(V_0)$, $Q_0(A_k v,A_k v)>0$ for $v\in\tilde l(V_0)\setminus\{0\}$.
Then, there exists $c\in(0,(N+1)/(N+3)]$ such that $Q_0(A_k v,A_k v)\ge c Q_0(v,v)$ for all $v\in\tilde l(V_0)$. This inequality also holds for any $v\in l(V_0)$.
This implies that, for any $v\in l(V_0)$,
\begin{align*}
Q_0(v,v)
&\ge\frac{N+3}{N+1}(Q_0(A_i v,A_i v)+Q_0(A_j v,A_j v)+Q_0(A_k v,A_k v))\\
&\ge\frac{N+3}{N+1}(Q_0(A_i v,A_i v)+Q_0(A_j v,A_j v)+cQ_0(v,v)),
\end{align*}
which implies that
\[
\frac{N+3}{N+1}(Q_0(A_i v,A_i v)+Q_0(A_j v,A_j v))
\le\left(1-\frac{N+3}{N+1}c\right)Q_0(v,v).
\]
Let $u=\iota^{-1}(h)$. Then, for $m\in\N$,
\begin{align*}
\nu_h(\psi_w(\overline{p_i p_j}))
&\le \sum_{w'\in\{i,j\}^m}\nu_h(K_{w w'})\\
&=2\left(\frac{N+3}{N+1}\right)^{|w|+m}\sum_{w'\in\{i,j\}^m}Q_0(A_{w'}A_w u,A_{w'}A_w u)
\quad\text{(from \Eq{cellenergy})}\\
&\le 2\left(\frac{N+3}{N+1}\right)^{|w|}\left(1-\frac{N+3}{N+1}c\right)^m Q_0(A_w u,A_w u).
\end{align*}
By letting $m\to\infty$, we obtain that $\nu_h(\psi_w(\overline{p_i p_j}))=0$.
\end{proof}

\section{Behaviors of harmonic functions on edges}\label{sec:location}
\subsection{Unimodality of harmonic functions on edges}
In this subsection, we fix a nonconstant harmonic function $h$.
For $i,j\in S$, we define the following.
\begin{align*}
p_{ij}&:=\frac12(p_i+p_j)\,(=\psi_i(p_j)=\psi_j(p_i)=p_{ji}),\\
\a_i&:=h(p_i),\\
\a_{ij}&:=h(p_{ij})\,(=h(\psi_i(p_j))=h(\psi_j(p_i))=\a_{ji}),\\
s&:=\frac1{N+1}\sum_{k\in S}\a_k,\\
s_i&:=\frac1{N+1}\sum_{k\in S}\a_{ik}.
\end{align*}
By writing $u={}^t(\a_1,\a_2,\dots,\a_{N+1})\in\R^{N+1}$, $\a_{ij}$ is equal to the $j$th component $(A_i u)_j$ of $A_i u$.
Then, 
\begin{align}
\a_{ij}&=\frac1{N+3}\{(N+1)s+\a_i+\a_j\}\quad\text{if $i\ne j$},\label{eq:aij}\\
s_i&=\frac{1}{N+1}\sum_{k=1}^{N+1}(A_i u)_k=\frac1{N+3}\{(N+1)s+2\a_i\}.\label{eq:si}
\end{align}
In particular,
\begin{align}
s_i-\a_i&=\frac{N+1}{N+3}(s-\a_i),\label{eq:siai}\\
\a_{ij}-s_i&=\frac1{N+3}(\a_j-\a_i),\label{eq:aijsi}\\
s_j-\a_j&=\frac{N+1}{N+3}(s-\a_j),\label{eq:sjaj}\\
\a_{ij}-s_j&=\frac1{N+3}(\a_i-\a_j).\label{eq:aijsj}
\end{align}

For $x,y\in K$, we often identify the line segment $\overline{xy}$ with the interval $[0,1]$ by the map $\Phi\colon[0,1]\to\overline{xy}$ defined as $\Phi(t)=(1-t)x+ty$.
We say that a function $f$ on $\overline{xy}$ is increasing if $f$ is increasing on $[0,1]$ by identifying $\overline{xy}$ with $[0,1]$ by $\Phi$.

In what follows, $i$ and $j$ are assumed to be distinct elements of $S$.
\begin{prop}\label{prop:5}
If $\a_i\le s\le \a_j$ and $\a_i<\a_j$, then $h$ is strictly increasing on $\overline{p_i p_j}$.
\end{prop}
\begin{proof}
From \Eq{siai}--\Eq{aijsj} and the assumption, it holds that $\a_i\le s_i<\a_{ij}<s_j\le \a_j$.
This means that the same assumptions hold for the functions $\psi_i^* h$ and $\psi_j^* h$.
Repeating this argument, we can show that $h$ is strictly increasing on the set $\{m/2^n\mid n\in\N,\ m=0,1,\dots,2^n\}\subset[0,1]\simeq\overline{p_i p_j}$. Since $h$ is continuous, the claim holds.
\end{proof}
\begin{cor}\label{cor:6}
\begin{enumerate}
\item If $\a_i=\a_j<s$, then $h$ attains its maximum on $\overline{p_i p_j}$ at $p_{ij}$.
\item If $s<\a_i=\a_j$, then $h$ attains its minimum on $\overline{p_i p_j}$ at $p_{ij}$.
\end{enumerate}
\end{cor}
\begin{proof}
\begin{enumerate}
\item From \Eq{siai}--\Eq{aijsj}, it holds that $\a_i<s_i=\a_{ij}$ and $\a_j<s_j=\a_{ji}$.
Therefore, we can apply \Prop{5} to $\psi_i^* h$ and $\psi_j^* h$ to obtain that $h$ is strictly increasing on $\overline{p_i p_{ij}}$ and on $\overline{p_j p_{ij}}$.
\item It suffices to apply (i) to $-h$ in place of $h$.
\qedhere
\end{enumerate}
\end{proof}
\begin{prop}\label{prop:7}
If $h$ is nondecreasing on $\overline{p_i p_j}$, then $s\le\a_j$.
\end{prop}
\begin{proof}
We first show that, for any $k\in\N$ and any $h\in\cH$ that is nondecreasing on $\overline{p_i p_j}$,
\begin{equation}\label{eq:kh}
\sum_{l=0}^k(N+1)^l\a_j\ge\a_i+\sum_{l=1}^k(N+1)^l s.
\end{equation}
When $k=1$, \Eq{kh} is equivalent to $\a_j\ge\frac1{N+3}\{(N+1)s+\a_i+\a_j\}=\a_{ij}$, which is true by assumption.

Suppose \Eq{kh} holds for some $k$. Applying \Eq{kh} to $\psi_j^* h$, we have
\begin{align*}
&\sum_{l=0}^k(N+1)^l\a_j\\
&\ge\a_{ij}+\sum_{l=1}^k(N+1)^l s_j\\
&=\frac1{N+3}\{(N+1)s+\a_i+\a_j\}+\frac1{N+3}\{(N+1)s+2\a_j\}\sum_{l=1}^k(N+1)^l\\
&\qquad\text{(by \Eq{aij} and \Eq{si} with $i$ replaced by $j$)}\\
&=\frac1{N+3}\a_i+\frac1{N+3}\left\{2\sum_{l=0}^k(N+1)^l-1\right\}\a_j+\frac{N+1}{N+3}\sum_{l=0}^k(N+1)^l s,
\end{align*}
which is equivalent to \Eq{kh} with $k$ replaced by $k+1$.
Therefore, \Eq{kh} holds by the mathematical induction.

By dividing both sides of \Eq{kh} by $\sum_{l=0}^k(N+1)^l$ and letting $k\to\infty$, we obtain that $\a_j\ge s$.
\end{proof}
In $\R^N$, let $H_{ij}$ denote the perpendicular bisecting hyperplane of the line segment $\overline{p_i p_j}$, and let $R_{ij}$ be the reflection across $H_{ij}$, mapping each point in  $\R^N$ to its symmetric point with respect to $H_{ij}$.
In particular, $R_{ij}(x)=p_i+p_j-x$ for $x\in\overline{p_i p_j}$.
It is easy to see that $h\circ R_{ij}\in\cH$ for $h\in\cH$.
\begin{prop}\label{prop:8}
Suppose $\a_i\le\a_j\le s$ and $\a_i<s$. Then, there exists a unique $x\in\overline{p_{ij} p_j}$ such that $h$ is strictly increasing on $\overline{p_i x}$ and strictly decreasing on $\overline{x p_j}$. Moreover, if $\a_j<s$ then $x\ne p_j$. If $\a_i<\a_j$ then $x\ne p_{ij}$.
\end{prop}
\begin{proof}
First, we give the following observation. Suppose $\a_i\le\a_j\le s$ and $\a_i<s$.
Let $p_{jij}$ denote the middle point of $\overline{p_j p_{ij}}$. 
From \Eq{sjaj} and \Eq{aijsj}, 
\begin{equation}\label{eq:asas}
\text{($\a_j<s_j$ and $\a_{ij}\le s_j$) or ($\a_j\le s_j$ and $\a_{ij}< s_j$)}
\end{equation}
holds.
From \Eq{siai} and \Eq{aijsi}, $\a_i<s_i\le\a_{ij}$ holds. From \Prop{5}, $\psi_i^* h$ is strictly increasing on $\overline{p_i p_j}$, that is, $h$ is strictly increasing on $\overline{p_i p_{ij}}$.

Then, there are two possibilities:
\begin{itemize}
\item If $\a_{ij}\le \a_j$, then $\a_{ij}\le\a_j\le s_j$ and $\a_{ij}<s_j$ from \Eq{asas}. Since the assumption of the above observation is satisfied for the function $\psi_j^* h$, we obtain that $\psi_j^* h$ is strictly increasing on $\overline{p_i p_{ij}}$, that is, $h$ is strictly increasing on $\overline{p_{ij} p_{jij}}$.
\item If $\a_{ij}>\a_j$, then $\a_j<\a_{ij}\le s_j$. Therefore, the assumption of the above observation is satisfied for the function $(\psi_j^* h)\circ R_{ij}$. Then $h$ is strictly decreasing on $\overline{p_{jij}p_j}$.
\end{itemize}

We now iterate this observation. 
We construct inductively closed subintervals $J_n\subset \overline{p_ip_j}$, $n\ge 0$, with $J_0=\overline{p_ip_j}$ and $J_{n+1}$ one of the two halves of $J_n$, together with maps $\theta_n\colon K\to K$, each being a composition of the maps $\psi_i,\psi_j$ and $R_{ij}$, such that $\theta_n$ maps $\overline{p_ip_j}$ bijectively onto $J_n$ and $g_n:=h\circ\theta_n$ satisfies the hypothesis of the observation.
Start with $\theta_0=\mathrm{id}$, $g_0=h$. Given $\theta_n$, the observation shows that $g_n$ is strictly increasing on $\overline{p_ip_{ij}}$, and that $g_{n+1}:=\psi_j^*g_n$ (in the first case) or $g_{n+1}:=(\psi_j^*g_n)\circ R_{ij}$ (in the second case) again satisfies the hypothesis. Set $J_{n+1}=\theta_n(\overline{p_{ij}p_j})$ and $\theta_{n+1}=\theta_n\circ\psi_j$ or $\theta_{n+1}=\theta_n\circ\psi_j\circ R_{ij}$ accordingly, so that $g_{n+1}=h\circ\theta_{n+1}$.

Write $J_n=\overline{a_nb_n}$ with $a_n$ closer to $p_i$. We claim that, for every $n$, $h$ is strictly increasing on $\overline{p_ia_n}$ and strictly decreasing on $\overline{b_np_j}$. This is vacuous for $n=0$. Assume it for $n$. If $\theta_n$ preserves orientation, then $\theta_n(p_i)=a_n$, and the strict increase of $g_n$ on $\overline{p_ip_{ij}}$ means that $h$ is strictly increasing on $\theta_n(\overline{p_ip_{ij}})=\overline{a_na_{n+1}}$, while $b_{n+1}=b_n$; if $\theta_n$ reverses orientation, then $\theta_n(p_i)=b_n$ and $h$ is strictly decreasing on $\theta_n(\overline{p_ip_{ij}})=\overline{b_{n+1}b_n}$, while $a_{n+1}=a_n$. In either case the claim for $n+1$ follows.

Since $|J_n|=2^{-n}|p_i-p_j|$, there is a unique $x$ with $\bigcap_nJ_n=\{x\}$; $a_n\uparrow x$, $b_n\downarrow x$, and $x\in J_1=\overline{p_{ij}p_j}$. 
By the claim, $h$ is strictly increasing on $\overline{p_ia_n}$ for every $n$, and $a_n\uparrow x$. We show that $h$ is strictly increasing on $\overline{p_ix}$. Let $y<y'$ be points of $\overline{p_ix}$. If $y'<x$, choose $n$ with $a_n\ge y'$; then $h(y)<h(y')$ because both points lie in $\overline{p_ia_n}$. If $y'=x$, choose $n$ with $a_n>y$. Then $h(y)<h(a_n)$, and $h(a_n)\le h(a_m)$ for all $m\ge n$; letting $m\to\infty$ and using the continuity of $h$, we obtain $h(a_n)\le h(x)$. Hence $h(y)<h(x)$. In the same way, $h$ is strictly decreasing on $\overline{xp_j}$. Such $x$ is unique, being the unique maximum point of $h$ on $\overline{p_ip_j}$.

If $x=p_j$, then $h$ is strictly increasing on $\overline{p_i p_j}$. From \Prop{7}, $s\le\a_j$. Therefore, $x\ne p_j$ if $\a_j<s$.
If $\a_i<\a_j$, \Eq{aijsj} implies $\a_{ij}<s_j$. 
Assume $x=p_{ij}$. Then $(\psi_j^*h)\circ R_{ij}$ is strictly increasing on $\overline{p_ip_j}$. From \Prop{7}, $s_j\le\a_{ij}$, which is a contradiction. Therefore, $x\ne p_{ij}$.
\end{proof}
\begin{cor}\label{cor:9}
If $s<\a_i<\a_j$ then there exists a unique $x\in\overline{p_i p_{ij}}\setminus\{p_i,p_{ij}\}$ such that $h$ is strictly decreasing on $\overline{p_i x}$ and strictly increasing on $\overline{x p_j}$.
\end{cor}
\begin{proof}
Apply \Prop{8} to the function $-h\circ R_{ij}$.
\end{proof}
\begin{thm}\label{th:10}
Suppose that $s\ne \a_i$ or $s\ne \a_j$.
Then, $\a_i=\a_j$ if and only if $h$ attains its maximum or minimum on $\overline{p_i p_j}$ at $p_{ij}$.
\end{thm}
\begin{proof}
The only if part follows from \Cor{6}.
We prove the if part.
Suppose $\a_i\ne \a_j$. By exchanging $i$ with $j$ if necessary, we may assume $\a_i<\a_j$. Then, in none of the cases $s<\a_i<\a_j$, $\a_i\le s\le \a_j$, and $\a_i<\a_j<s$ does the function $h$ attain its maximum at $p_{ij}$ from \Cor{9}, \Prop{5}, and \Prop{8}.
By considering $-h$ in place of $h$, $h$ does not also attain its minimum at $p_{ij}$.
\end{proof}
We refer the reader to \cite{DSV99,QT19} for related results concerning this subsection in the case $N=2$.
\subsection{The locations of extremal points of a harmonic function along edges}
From \Eq{aij} and \Eq{si}, $\a_{ij}$, $s_i$, and $s_j$ are determined by only $\a_i$, $\a_j$, and $s$.
Repeating this argument, we can see that the values of $h$ on $\{m/2^n\mid n\in\N,\ m=0,1,\dots,2^n\}\subset[0,1]\simeq\overline{p_ip_j}$ are determined by $\a_i$, $\a_j$, and $s$.
Since $h$ is continuous, the values of $h$ on $\overline{p_ip_j}$ are determined by $\a_i$, $\a_j$, and $s$.

Suppose $\a_i\ne\a_j$. From \Prop{5}, \Prop{8}, and \Cor{9}, there exists a unique point attaining the maximum of $h$ on $\overline{p_ip_j}$. Such a point will be denoted by $M(s;\a_i,\a_j)\in[0,1]\simeq\overline{p_ip_j}$. We set $M(s):=M(s;-1,0)$. Then, it is easy to see that 
\begin{equation}\label{eq:M}
M(s;\a_i,\a_j)=\begin{cases}
\displaystyle M\left(\frac{s-\a_j}{\a_j-\a_i}\right) &\text{if $\a_i<\a_j$}\\
\displaystyle 1-M\left(\frac{s-\a_i}{\a_i-\a_j}\right) &\text{if $\a_i>\a_j$}.
\end{cases}
\end{equation}
From \Prop{5} and \Cor{9}, $M(s)=1$ for $s\le 0$.
In what follows, we assume $s>0$. Then $M(s)\in(1/2,1)$ by \Prop{8}. In particular, $h$ attains its maximum $M(s)$ on $[1/2,1]\simeq\overline{p_{ij}p_j}$.
\begin{prop}\label{prop:fe}
$M(s)$ satisfies the following functional equation:
\[
M(s)=\begin{cases}
\displaystyle\frac12+\frac12 M\left(\frac{(N+1)s}{1-(N+1)s}\right) & \displaystyle\text{if $s\in\left(0,\frac1{N+1}\right)$}\smallskip\\
\displaystyle\frac34 & \displaystyle\text{if $s=\frac1{N+1}$}\smallskip\\
\displaystyle 1-\frac12 M\left(\frac{1}{(N+1)s-1}\right) & \displaystyle\text{if $s\in\left(\frac1{N+1},\infty\right)$}.
\end{cases}
\]
Moreover,
\[
\left\{
\begin{alignedat}{1}
s \in \left(0,\frac{1}{N+1}\right)
&\iff M(s) \in \left(\frac{3}{4},1\right), \\
s = \frac{1}{N+1}
&\iff  M(s) = \frac{3}{4}, \\
s \in \left(\frac{1}{N+1},\infty\right)
&\iff  M(s) \in \left(\frac{1}{2},\frac{3}{4}\right).
\end{alignedat}
\right.
\]
\end{prop}
\begin{proof}
Consider a harmonic function $h$ such that $\a_i=-1$, $\a_j=0$, and $s>0$.
From \Eq{aij} and \Eq{si} with $i$ replaced by $j$, it holds that
\[
\a_{ij}=\frac1{N+3}\{(N+1)s-1\}\text{  and }s_j=\frac{N+1}{N+3}s.
\]
In the case $\a_{ij}<0\,({\iff} s<1/(N+1))$, we have
\begin{align*}
M(s)&=\frac12+\frac12M(s_j;\a_{ij},0)\\
&=\frac12+\frac12 M\left(-\frac{s_j}{\a_{ij}}\right)\qquad\text{(from \Eq{M})}\\
&=\frac12+\frac12 M\left(\frac{(N+1)s}{1-(N+1)s}\right)
\in \left(\frac34,1\right).
\end{align*}
In the case $\a_{ij}>0\,({\iff} s>1/(N+1))$, we have
\begin{align*}
M(s)&=\frac12+\frac12M(s_j;\a_{ij},0)\\
&=1-\frac12 M\left(\frac{s_j}{\a_{ij}}-1\right)\qquad\text{(from \Eq{M})}\\
&=1-\frac12 M\left(\frac{1}{(N+1)s-1}\right)
\in \left(\frac12,\frac34\right).
\end{align*}
In the case $\a_{ij}=0\,({\iff} s=1/(N+1))$, we have
$s_j=1/(N+3)>\a_{ij}=\a_j=0$.
From \Cor{6}, $\psi_j^* h$ attains its maximum on $\overline{p_ip_j}$ at $p_{ij}$. This means that $M(s)=3/4$.
\end{proof}
Let $\ph$ be a homeomorphism from $(0,1)$ to $(0,\infty)$ defined as 
\[
 \ph(t)=\frac1{N+1}\frac{1-t}{t}.
\]
Then, $\ph^{-1}(t)=1/\{(N+1)t+1\}$.
Let $L=M\circ \ph$. For $t\in(0,1/2)$,
\begin{align*}
L(t)&=1-\frac12M\left(\frac1{\frac{1-t}{t}-1}\right)\quad\left(\text{since $\ph(t)>\frac1{N+1}$}\right)\\
&=1-\frac12 M\left(\frac{t}{1-2t}\right)\\
&=1-\frac12 L\left(\frac{1}{(N+1)\frac{t}{1-2t}+1}\right)\\
&=1-\frac12 L\left(\frac{1-2t}{(N-1)t+1}\right).
\end{align*}
For $t\in(1/2,1)$,
\begin{align*}
L(t)&=\frac12+\frac12M\left(\frac{\frac{1-t}{t}}{1-\frac{1-t}{t}}\right)\quad\left(\text{since $\ph(t)<\frac1{N+1}$}\right)\\
&=\frac12+\frac12 M\left(\frac{1-t}{2t-1}\right)\\
&=\frac12+\frac12 L\left(\frac{1}{(N+1)\frac{1-t}{2t-1}+1}\right)\\
&=\frac12+\frac12 L\left(\frac{2t-1}{(N-1)(1-t)+1}\right).
\end{align*}
Also, $L(1/2)=3/4$. By letting $L(0)=1/2$ and $L(1)=1$, $L$ is a mapping from $[0,1]$ to $[1/2,1]$ and satisfies the following functional equation:
\begin{equation}\label{eq:L}
L(t)=\begin{cases}
\displaystyle 1-\frac12 L\left(\frac{1-2t}{(N-1)t+1}\right) & \text{if $0\le t\le \dfrac12$}\\[10pt]
\displaystyle \frac12+\frac12 L\left(\frac{2t-1}{(N-1)(1-t)+1}\right) & \text{if $\dfrac12\le t\le 1$}.
\end{cases}
\end{equation}
Define contractions
\[
f_1(y):=1-\frac12 y,\qquad 
f_2(y):=\frac12+\frac12 y \qquad \left(y\in \left[\frac12,1\right]\right),
\]
and piecewise monotone maps
\[
g_1(t):=\frac{1-2t}{(N-1)t+1}\ \left(0\le t\le \frac12\right),\quad
g_2(t):=\frac{2t-1}{(N-1)(1-t)+1}\ \left(\frac12\le t\le 1\right).
\]
\begin{lem}\label{lem:L}
$L$ is continuous and nondecreasing on $[0,1]$.
\end{lem}
\begin{proof}
Equation~\Eq{L} is a de Rham type functional equation (cf.\ \cite[Section~6]{Ha85}); we give a direct argument.
Let $\cB$ be the set of all functions $F\colon[0,1]\to[1/2,1]$ with $F(0)=1/2$ and $F(1)=1$, equipped with the uniform metric $d(F,G)=\sup_{t\in[0,1]}|F(t)-G(t)|$. $\cB$ is complete.
For $F\in\cB$ define
\[
\Xi F=\begin{cases}
f_1\circ F\circ g_1&\text{on }[0,1/2]\\
f_2\circ F\circ g_2&\text{on }[1/2,1].\end{cases}
\]
Since $g_1(1/2)=g_2(1/2)=0$ and $f_1(F(0))=f_2(F(0))=3/4$, the two expressions agree at $t=1/2$, so $\Xi$ is well-defined.
Moreover, $(\Xi F)(0)=f_1(F(g_1(0)))=f_1(F(1))=1/2$ and $\Xi F(1)=f_2(F(g_2(1)))=f_2(F(1))=1$, so $\Xi(\cB)\subset\cB$.
As $f_1$ and $f_2$ are contractions with ratio $1/2$, we have $d(\Xi F,\Xi G)\le (1/2)d(F,G)$, and $\Xi$ has a unique fixed point in $\cB$.
Now $L\in\cB$, since $L$ takes values in $[1/2,1]$, and \Eq{L} states precisely that $\Xi L=L$. Hence $L$ is this unique fixed point.

Let $\cC=\cB\cap C([0,1])$, which is a closed subset of $\cB$. If $F\in\cC$, then $\Xi F$ is continuous on each of $[0,1/2]$ and $[1/2,1]$, and the two pieces agree at $t=1/2$. 
Thus $\Xi(\cC)\subset \cC$.
Put $F_0(t):=(1+t)/2\in\cC$. Then $\Xi^n F_0\in\cC$ for every $n$ and $\Xi^n F_0\to L$ uniformly, so $L\in\cC$; in particular $L$ is continuous.
Finally, if $F$ is nondecreasing, so is $\Xi F$. By induction each $\Xi^n F_0$ is nondecreasing, and hence so is the uniform limit $L$.
\end{proof}
To prove that $L$ is in fact strictly increasing, we first show that $L$ admits the representation $L=\chi\circ v_G$, where $v_G\colon[0,1]\to \{1,2\}^{\Z_+}$ and $\chi\colon\{1,2\}^{\Z_+}\to[1/2,1]$ are defined as follows.
$v_G$ is the itinerary map of the associated piecewise monotone interval map
$G\colon [0,1]\to[0,1]$ defined as
\[
G(t)=
\begin{cases}
g_1(t) & \text{for $t\in[0,1/2]$}\\
g_2(t) & \text{for $t\in(1/2,1]$.}
\end{cases}
\]
That is, $v_G(t)=(\om_0(t),\om_1(t),\om_2(t),\ldots)$, where, for each $n\ge0$,
\[
\om_n(t):=
\begin{cases}
1 & \text{if $G^{n}(t)\in[0,1/2]$}\\
2 & \text{if $G^{n}(t)\in(1/2,1]$.}
\end{cases}
\]
Here, $G^{n}$ denotes the $n$ times iteration of $G$ and $G^0$ denotes the identity map.
$\chi$ is the coding map of the iterated function system
$\{f_1,f_2\}$, defined as
\[
\chi(\om_0,\om_1,\ldots)
:=\lim_{n\to\infty}
f_{\om_0}\circ f_{\om_1}\circ\cdots\circ f_{\om_{n-1}}(1).
\]
The coding map $\chi$ is well defined and continuous by the general theory of iterated function systems (see, e.g., \cite{MT88}).
Since $\chi(\om)=f_{\om_0}(\chi(\sg\om))$ and $L(t)=f_{\om_0(t)}(L(G(t)))$ by \Eq{L}, where
\[
\sg\colon \{1,2\}^{\Z_+}\ni \om=(\om_0,\om_1,\ldots)\mapsto(\om_1,\om_2,\ldots)\in \{1,2\}^{\Z_+}
\]
is a shift operator, iterating $n$ times gives
\[
L(t)=f_{\om_0(t)}\circ\cdots\circ f_{\om_{n-1}(t)}\bigl(L(G^n(t))\bigr),\quad \chi(v_G(t))=f_{\om_0(t)}\circ\cdots\circ f_{\om_{n-1}(t)}\bigl(\chi(\sg^n v_G(t))\bigr).
\]
Both $L(G^n(t))$ and $\chi(\sg^nv_G(t))$ lie in $[1/2,1]$, and each $f_k$ is a contraction with ratio $1/2$, so $|L(t)-\chi(v_G(t))|\le 2^{-n-1}$ for every $n$. Hence $L=\chi\circ v_G$.

We consider inverse branches
\[
g_1^{-1}(s)=\frac{-s+1}{(N-1)s+2}\in\Bigl[0,\frac12\Bigr],\qquad
g_2^{-1}(s)=\frac{Ns+1}{(N-1)s+2}\in\Bigl[\frac12,1\Bigr].
\]
For $\om=(\om_0,\om_1,\ldots)\in\{1,2\}^{\Z_+}$ and $n\ge0$ set
\[
I_{[\om]_n}=g_{\om_0}^{-1}(g_{\om_1}^{-1}(\cdots (g_{\om_{n-1}}^{-1}([0,1]))\cdots)).
\]
\begin{lem}\label{lem:I}
For any $\om\in\{1,2\}^{\Z_+}$, the diameter of $I_{[\om]_n}$ converges to $0$ as $n\to\infty$.
\end{lem}
\begin{proof}
$g_1^{-1}$ and $g_2^{-1}$ are both Lipschitz continuous.
From the direct calculation, for $k=1,2$,
\begin{align*}
\left|\frac{\dd}{\dd s}(g_k^{-1}\circ g_1^{-1})\right|&=\frac{(N+1)^2}{\{N+3+(N-1)s\}^2}\le \left(\frac{N+1}{N+3}\right)^2<1,\\
\left|\frac{\dd}{\dd s}(g_k^{-1}\circ g_2^{-1})\right|&=\frac{(N+1)^2}{\{N+3+(N^2+N-2)s\}^2}\le \left(\frac{N+1}{N+3}\right)^2<1.
\end{align*}
Therefore, the diameter of $I_{[\om]_n}$ converges to $0$ exponentially fast.
\end{proof}
Let
\[
B=\left\{t\in[0,1]\mathrel{}\middle|\mathrel{} G^{n}(t)\ne\frac12\ \text{for all }n\ge0\right\}.
\]
The set $[0,1]\setminus B$ is countable. In particular, $B$ is dense in $[0,1]$.
\begin{lem}\label{lem:injective_itinerary}
If $s,t\in B$ and $v_G(s)=v_G(t)$, then $s=t$.
\end{lem}

\begin{proof}
For $\om=v_G(s)$, both $s$ and $t$ belong to $\bigcap_{n\in\Z_+}I_{[\om]_n}$. From \Lem{I}, $s=t$ holds.
\end{proof}

\begin{lem}\label{lem:B}
If two distinct elements $\om$ and $\hat\om$ of $\{1,2\}^{\Z_+}$ satisfy $\chi(\om)=\chi(\hat\om)$, then there exist $n\in\Z_+$ and $w\in\{1,2\}^n$ such that $\{\om,\hat\om\}=\{w112^\infty,w212^\infty\}$.
\end{lem}
\begin{proof}
We can express $\om=w k \om'$ and $\hat\om=w l \hat\om'$ for some $n\in\Z_+$, $w\in\{1,2\}^n$, $\{k,l\}=\{1,2\}$, and $\om',\hat\om'\in \{1,2\}^{\Z_+}$.
We may assume $(k,l)=(1,2)$.
The affine maps $f_1$ and $f_2$ are injective, hence $\chi(\om)=\chi(\hat\om)$ implies $\chi(1 \om')=\chi(2 \hat\om')$.
Writing $a=\chi(\om')$ and $b=\chi(\hat\om')$, this equality reads $f_1(a)=f_2(b)$, that is, $a+b=1$. Since $a,b\in[1/2,1]$, it holds that $a=b=1/2$.

Now observe that the equation $\chi(\eta)=1/2$ for $\eta\in \{1,2\}^{\Z_+}$ implies $\eta=12^\infty$.
Indeed, the equation $1/2=\chi(\eta)=f_{\eta_0}(\chi(\sg\eta))$ forces $\eta_0=1$ and $\chi(\sg\eta)=1$. Since $f_1^{-1}(\{1\})=\emptyset$ and $f_2^{-1}(\{1\})=\{1\}$, $\sg\eta$ must be $2^\infty$.
This proves the claim.
\end{proof}

\begin{lem}\label{lem:boundary_H}
If $v_G(t)=w'12^\infty$ for some $m\in\N$ and $w'\in\{1,2\}^m$, then $G^{m-1}(t)=1/2$.
In particular, $t\notin B$.
\end{lem}
\begin{proof}
Write $w'=w i$ with $w\in\{1,2\}^{m-1}$ and $i\in\{1,2\}$.
Since $G^{m-1}(t)$ has itinerary $i12^\infty$,
\[
G^{m-1}(t)\in \bigcap_{k\in\N}g_i^{-1}(g_1^{-1}((g_2^{-1})^k([0,1])))=g_i^{-1}(g_1^{-1}(\{1\}))=g_i^{-1}(\{0\})=\left\{\frac12\right\}.
\]
Therefore, $G^{m-1}(t)=1/2$.
\end{proof}

\begin{prop}\label{prop:strict_increasing}
$L$ is strictly increasing on $[0,1]$.
\end{prop}
\begin{proof}
Suppose $s,t\in B$ satisfy $s<t$.
From \Lem{injective_itinerary}, $v_G(s)\neq v_G(t)$.
Moreover, Lemmas~\ref{lem:B} and \ref{lem:boundary_H} exclude the possibility
$\chi(v_G(s))=\chi(v_G(t))$.
Since $L=\chi\circ v_G$ is nondecreasing from \Lem{L}, this implies $L(s)<L(t)$.
Because $B$ is dense in $[0,1]$ and $L$ is continuous from \Lem{L}, the strict inequality extends to all $s<t$ in $[0,1]$.
\end{proof}
\Prop{strict_increasing} and the relation $M=L\circ \ph^{-1}$ imply the following.
\begin{prop}\label{prop:M}
The map $M\colon(0,\infty)\to(1/2,1)$ is a strictly decreasing and continuous surjection.
\end{prop}
\begin{rem}
For $N = 2$, \Prop{strict_increasing} is mentioned in \cite[Section~4, Remark~2]{DSV99}, where a proof of the strict monotonicity is omitted. Since this property is needed in the proof of \Thm{main1a} in the next section, we provide a proof here valid for arbitrary $N \ge 2$.
\end{rem}
\section{Proof of Theorem~\protect\ref{th:main1}}\label{sec:main1}
Let $w\in W_*$. We say that a harmonic function $h$ is \emph{symmetric} on $K_w$ if there exist distinct elements $i$ and $j$ of $S$ such that $h(\psi_w(p_i))=h(\psi_w(p_j))$.

The following theorem is a generalization of a part of \cite[Theorem~3.5]{BHS14} to general $N$.
\begin{thm}\label{th:main1a}
Let $h\in \cH$ and $w\in W_*$. Then,
\begin{equation}\label{eq:main1}
\inf_{w'\in W_*}\frac{\nu_h(K_{ww'})}{\nu(K_{ww'})}=0.
\end{equation}
\end{thm}
\begin{proof}
The proof is divided into two steps.
First we treat the symmetric case, and then reduce the general case to it by a perturbation argument.
\smallskip

\noindent
\textbf{Step 1.} We consider the case where $h$ is symmetric on $K_w$. Take distinct elements $i$ and $j$ of $S$ such that $h(\psi_w(p_i))=h(\psi_w(p_j))$.
Let $u=\iota^{-1}(\psi_w^* h)\in l(V_0)$. Then, $u(p_i)=u(p_j)$.
Since
\[
u=u(p_i)\bfone+\sum_{k\in S\setminus\{i,j\}}(u(p_k)-u(p_i))e_k,
\]
we have
\[
A_i u=u(p_i)\bfone+\frac1{N+3}\sum_{k\in S\setminus\{i,j\}}(u(p_k)-u(p_i))(\bfone+e_k-e_i).
\]
Because $e_k-e_i\in E_j$ for $k\in S\setminus\{i,j\}$ (recall \Eq{Ek}), \Lem{eigen} implies for $n\in\N$ that
\begin{align}\label{eq:AjAi}
A_j^nA_i u&=u(p_i)\bfone+\frac1{N+3}\sum_{k\in S\setminus\{i,j\}}(u(p_k)-u(p_i))\bfone\\
&\quad+\left(\frac1{N+3}\right)^{n+1}\sum_{k\in S\setminus\{i,j\}}(u(p_k)-u(p_i))(e_k-e_i).\nonumber
\end{align}
From \Eq{ww'}, \Eq{cellenergy}, and \Eq{AjAi},
\begin{align*}
\nu_h(K_{w i j^n})
&=\left(\frac{N+3}{N+1}\right)^{|w|}\nu_{\psi_w^* h}(K_{i j^n})\\
&=2\left(\frac{N+3}{N+1}\right)^{|w|+n+1}\cE(\psi_{i j^n}^*\psi_w^* h,\psi_{i j^n}^*\psi_w^* h)\\
&=2\left(\frac{N+3}{N+1}\right)^{|w|+n+1}Q_0(A_j^nA_i u,A_j^nA_i u)\\
&=2\left(\frac{N+3}{N+1}\right)^{|w|+n+1}\left(\frac{1}{N+3}\right)^{2(n+1)}\\
&\quad\times\sum_{k,l\in S\setminus\{i,j\}}(u(p_k)-u(p_i))(u(p_l)-u(p_i))Q_0(e_k-e_i, e_l-e_i)\\
&=O\left(\left\{\frac1{(N+1)(N+3)}\right\}^n\right)\qquad (n\to\infty).
\end{align*}
By combining \Prop{4}, we obtain that
\[
\frac{\nu_h(K_{w i j^n})}{\nu(K_{w i j^n})}=O\biggl(\left(\frac1{N+1}\right)^{2n}\biggl)\qquad (n\to\infty),
\]
which implies \Eq{main1}.
\smallskip

\noindent
\textbf{Step 2.} We consider the general case.
Let $\b_1\le\b_2\le\cdots\le \b_{N+1}$ be a rearrangement of $h(\psi_w(p_1)), h(\psi_w(p_2)),\dots,h(\psi_w(p_{N+1}))$.
If $\b_i=\b_{i+1}$ for some $i$, then \Eq{main1} holds by Step~1.
We will assume $\b_i<\b_{i+1}$ for all $i$.
Let $s=(N+1)^{-1}\sum_{i\in S}h(\psi_w(p_i))$.
Then, $\b_1<s<\b_{N+1}$.
There are three possibilities:
\smallskip

\noindent(i) The case $(\b_1<)\,\b_2<s$.
Take $i,j\in S$ such that
$\b_1=h(\psi_w(p_i))$ and $\b_2=h(\psi_w(p_j))$.
From \Prop{8}, $h$ attains its maximum on
$\psi_w(p_ip_j)$ at some point in
$\psi_w(p_ip_j)\setminus\{\psi_w(p_i),\psi_w(p_j)\}$.
That is, $M(s;\b_1,\b_2)\in(1/2,1)$.

Fix $l\in S\setminus\{i,j\}$ and define
$\hat h=\iota(A_w^{-1}e_l)\in\cH$.
We note that $A_w$ is invertible and
$\psi_w^*\hat h=\iota(e_l)$.

Let $\eps>0$.
From \Prop{M} and \Eq{M}, together with
the density of odd dyadic rationals in $(1/2,1)$,
we can take $s'>\b_2$ such that
$|s-s'|<\eps$ and $M(s';\b_1,\b_2)={m}/{2^n}$ for some $n\in\N$ and $m\in\{1,3,5,\ldots,2^n-1\}$.

Define $\Phi_w:[0,1]\to\psi_w(p_ip_j)$ by $\Phi_w(t)=\psi_w((1-t)p_i+tp_j)$.
Let $q=\Phi_w\bigl((m-1)/2^n\bigr)$ and $r=\Phi_w\bigl((m+1)/2^n\bigr)$.
Define $h'=h+(N+1)(s'-s)\hat h$.
Then
\[
  (\psi_w^*h')(p_i)=\b_1,
  \quad
  (\psi_w^*h')(p_j)=\b_2,
  \quad\text{and}\quad
  \frac{1}{N+1}\sum_{k\in S}(\psi_w^*h')(p_k)=s'.
\]
Thus, $h'$ attains its maximum on
$\psi_w(p_ip_j)$ at
$\Phi_w(m/2^n)$, which is the midpoint
of $qr$.
Since $m$ is odd, $qr$ is a level-$(n-1)$
dyadic subedge of $\psi_w(p_ip_j)$.
Hence there exists $\hat w\in\{i,j\}^{n-1}$
such that $q=\psi_{w\hat w}(p_i)$ and $r=\psi_{w\hat w}(p_j)$.
Applying \Thm{10} to $\psi_{w\hat w}^*h'$,
we obtain $h'(q)=h'(r)$.
(If the hypothesis of \Thm{10} is not
satisfied, then $h'(q)=h'(r)$ is immediate.)
From the result of Step~1,
\begin{equation}\label{eq:w''}
  \inf_{w''\in W_*}\frac{\nu_{h'}(K_{w\hat w w''})}{\nu(K_{w\hat w w''})}=0.
\end{equation}
On the other hand, for any Borel set $A$ of $K$,
\begin{align*}
\sqrt{\nu_h(A)}&\le \sqrt{\nu_{h'}(A)}+\sqrt{\nu_{h-h'}(A)}
\qquad\text{(from e.g., \cite[(2.1)]{Hi10})}\\
&=\sqrt{\nu_{h'}(A)}+(N+1)|s'-s|\sqrt{\nu_{\hat h}(A)}.
\end{align*}
Combining this with \Eq{w''} and \Lem{simple}, we obtain that
\[
\inf_{w''\in W_*}\sqrt{\frac{\nu_{h}(K_{w\hat w w''})}{\nu(K_{w\hat w w''})}}\le (N+1)\eps\sqrt{\nuesssup_{x\in K_{w\hat w}}\frac{\dd \nu_{\hat h}}{\dd \nu}(x)}
\le (N+1)|P\iota^{-1}(\hat h)|\eps.
\]
Since $\eps>0$ is arbitrary, we obtain \Eq{main1}.
\smallskip

\noindent(ii) The case of $s<\b_{N}\,(\,<\b_{N+1})$. Take $i,j\in S$ such that $\b_{N+1}=h(\psi_w(p_i))$ and $\b_{N}=h(\psi_w(p_j))$. From \Prop{8}, $-h$ attains its maximum on $\psi_w(\overline{p_ip_j})$ at some point in $\psi_w(\overline{p_ip_j})\setminus\{\psi_w(p_i),\psi_w(p_j)\}$.
Then, we obtain \Eq{main1} by the same argument as above case.
\smallskip

\noindent(iii) The other possibility is only when $N=2$ and $\b_1<\b_2=s<\b_3$. In such a case, $\b_2-\b_1=\b_3-\b_2$. 
Let $i,j,k\in S$ be taken so that $\b_1=h(\psi_w(p_i))$, $\b_2=h(\psi_w(p_j))$, and $\b_3=h(\psi_w(p_k))$. Let $u=\iota^{-1}(\psi_w^* h)\in l(V_0)$. Then, $u$ is described as $u=\b_2\bfone+(\b_2-\b_1)(e_k-e_i)$ so that for $n\in\N$,
\[
A_j^n u=\b_2\bfone+\left(\frac1{N+3}\right)^n(\b_2-\b_1)(e_k-e_i)
\]
from \Lem{eigen}.
From \Eq{cellenergy},
\begin{align*}
\nu_h(K_{w j^n})
&=2\left(\frac{N+3}{N+1}\right)^{|w|+n}\cE(\psi_{j^n}^*\psi_w^* h,\psi_{j^n}^*\psi_w^* h)\\
&=2\left(\frac{N+3}{N+1}\right)^{|w|+n}Q_0(A_j^n u, A_j^n u)\\
&=2\left(\frac{N+3}{N+1}\right)^{|w|+n}\left(\frac{1}{N+3}\right)^{2n}(\b_2-\b_1)^2Q_0(e_k-e_i,e_k-e_i)\\
&=O\left(\left\{\frac1{(N+1)(N+3)}\right\}^n\right)\qquad (n\to\infty).
\end{align*}
By combining this with \Prop{4}, we obtain that
\[
\frac{\nu_h(K_{w j^n})}{\nu(K_{w j^n})}=O\biggl(\left(\frac1{N+1}\right)^{2n}\biggl)\qquad (n\to\infty),
\]
which implies \Eq{main1}.
\end{proof}
\begin{cor}\label{cor:main1b}
Let $h\in\cH$ and $w\in W_*$.
Then,
\[
\nuessinf_{x\in K_w}\frac{\dd\nu_h}{\dd\nu}(x)=0.
\]
\end{cor}
\begin{proof}
Let $\dl:=\nuessinf_{x\in K_w}({\dd\nu_h}/{\dd\nu})(x)$.
Then, for any $w'\in W_*$,
\[
\nu_h(K_{w w'})=\int_{K_{w w'}}\frac{\dd\nu_h}{\dd\nu}\,\dd\nu\ge \dl\nu(K_{w w'}).
\]
\Thm{main1a} implies that $\dl$ has to be $0$.
\end{proof}
If $h\in\cH$ is a nonconstant function, by \cite[Theorem~5.6]{Hi10},
\begin{equation}\label{eq:Hi10}
\frac{\dd\nu_h}{\dd\nu}>0\quad\nu\text{-a.e.}
\end{equation}
In particular, for any $w\in W_*$
\begin{equation}\label{eq:esssup}
\nuesssup_{x\in K_w}\frac{\dd\nu_h}{\dd\nu}(x)>0.
\end{equation}
\begin{proof}[Proof of \Thm{main1}]
First, we note that $\nu(V_*)=0$.
Let $x\in K\setminus V_*$. For $n\in\Z_+$, let $[x]_n$ denote the unique element of $W_n$ such that $x\in K_{[x]_n}$. Then, $K_{[x]_n}\supset K_{[x]_{n+1}}$ for $n\in\Z_+$ and $\bigcap_{n=0}^\infty K_{[x]_n}=\{x\}$.
Define
\[
f(x)=\limsup_{n\to\infty}\frac{\nu_h(K_{[x]_n})}{\nu(K_{[x]_n})},\qquad x\in K\setminus V_*.
\]
From the martingale convergence theorem, $f=\dd\nu_h/\dd\nu$ $\nu$-a.e.
Let $\hat K=\{x\in K\setminus V_*\mid 0<f(x)<\infty\}$. Then $\nu(K\setminus \hat K)=0$ from \Eq{Hi10}.
Let $g$ be an arbitrary Borel $\nu$-version of $\dd\nu_h/\dd\nu$ and $x\in \hat K$.
Since $K_{[x]_{n+1}}\subset K_{[x]_n}$ for $n\in \Z_+$, the sequence $\{\nuesssup_{y\in K_{[x]_n}}g(y)\}_{n=0}^\infty$ is nonincreasing, so it has a limit $\gm$. Further, it holds that
\begin{equation}\label{eq:gm}
\gm\ge f(x)\,(>0).
\end{equation}
Indeed, if $\gm<f(x)$, there exist some $\eps>0$ and $n_0\in\N$ such that for any $n\ge n_0$, $\nuesssup_{y\in K_{[x]_n}}g(y)\le f(x)-\eps$.
Then
\[
\frac{\nu_h(K_{[x]_n})}{\nu(K_{[x]_n})}
=\frac{1}{\nu(K_{[x]_n})}\int_{K_{[x]_n}}g(y)\,\nu(\dd y)
\le f(x)-\eps.
\]
By taking $\limsup_{n\to\infty}$ we obtain $f(x)\le f(x)-\eps$, which is a contradiction.

On the other hand, from \Cor{main1b},
\begin{equation}\label{eq:g}
\nuessinf_{y\in K_{[x]_n}}g(y)=0,\qquad n\in\Z_+.
\end{equation}
By noting $\nu(K_{[x]_n})>0$ for all $n\in\Z_+$, \Eq{gm} and \Eq{g} imply that there exist two sequences $\{x_n\}$ and $\{x'_n\}$ converging to $x$ such that $\liminf_{n\to\infty}g(x_n)\ge f(x)>0$ and $\lim_{n\to\infty} g(x'_n)=0$. This means that $g$ is discontinuous at $x$.
\end{proof}
\begin{rem}\label{rem:comment}
When $N=2$, \cite[Theorem~3.5]{BHS14} shows that, for any
$w\in W_*$,\break
$\nuesssup_{x\in K_w}\frac{\dd\nu_h}{\dd\nu}(x)$
equals a positive constant depending on $h$, but independent of
$w$.  Since the space of harmonic functions modulo constants is
two-dimensional in this case, this fact also follows from
\Cor{main1b}.  Combined with the argument in the proof of
\Thm{main1}, this shows that we may take $\hat K=K$ when
$N=2$; that is, every Borel $\nu$-version of $\dd\nu_h/\dd\nu$ is
discontinuous at every point of $K$.
For $N\ge3$, only \Eq{esssup} is known, and accordingly the
exceptional set $K\setminus\hat K$ in \Thm{main1} may be
nontrivial.  Whether this exceptional set can be removed remains open.
\end{rem}

\section{Proof of Theorem~\protect\ref{th:main3}}\label{sec:main3}
We fix $i,j\in S$ with $i\ne j$.
Let $Y$ be a two-dimensional subspace of $\tilde l(V_0)$ that is defined as
\[
    Y=\{u\in\tilde l(V_0)\mid \text{$u$ is constant on $V_0\setminus\{p_i,p_j\}$}\}.
\]
Let $Z$ be the orthogonal complement of $Y$ in $\tilde l(V_0)$, that is,
\[
  Z=\{u\in\tilde l(V_0)\mid u(p_i)=u(p_j)=0\}.
\]
We often identify $Y$ with $\R^2$ by $u\mapsto (u(p_i),u(p_j))$, in other words, identify $\a u_1+\b u_2\in Y$ for $\a,\b\in\R$ with $(\a,\b)\in\R^2$, where
\[
u_1(p)=\begin{cases}
1& (p=p_i)\\
0& (p=p_j)\\
-\frac{1}{N-1}& (p\in V_0\setminus\{p_i,p_j\}),
\end{cases}
\quad
u_2(p)=\begin{cases}
0& (p=p_i)\\
1& (p=p_j)\\
-\frac{1}{N-1}& (p\in V_0\setminus\{p_i,p_j\}).
\end{cases}
\]
For $k\in S$, let $T_k=P A_k P$, which is regarded as a bounded operator on $\tilde l(V_0)$ as well as on $l(V_0)$.
\begin{lem}\label{lem:invariant}
The following hold.
\begin{enumerate}
\item $Y$ is an invariant space with respect to both $T_i$ and $T_j$. The representation matrix $L_i$ (resp.\ $L_j$) of $T_i$ (resp.\ $T_j$) on $Y\simeq \R^2$ is given by
\begin{equation}\label{eq:expression_L}
L_i=\frac1{N+3}\begin{pmatrix} N+1 & 0\\-1& 1\end{pmatrix}
\quad\text{and}\quad
L_j=\frac1{N+3}\begin{pmatrix} 1 & -1\\0& N+1\end{pmatrix},
\end{equation}
respectively.
\item $Z$ is an invariant space with respect to both $T_i$ and $T_j$, and $T_i|_Z=T_j|_Z=\frac1{N+3}I_Z$, where $I_Z$ is the identity operator on $Z$.
\end{enumerate}
\end{lem}
\begin{proof}
(i) For $T_i$, it suffices to prove that 
\[
T_i u_1=\frac1{N+3}\{(N+1)u_1-u_2\}\quad\text{and}\quad T_i u_2=\frac1{N+3}u_2.
\]
For notational simplicity, we prove this for $i=1$ and $j=2$.
By identifying $l(V_0)$ with $\R^{N+1}$, we have
\begin{align*}
A_1 u_1&=\frac1{N+3}\prescript{t\!}{}{\left(N+3,1,\frac{N-2}{N-1},\dots,\frac{N-2}{N-1}\right)},\\
A_1 u_2&=\frac1{N+3}\prescript{t\!}{}{\left(0,1,-\frac{1}{N-1},\dots,-\frac{1}{N-1}\right)}
\end{align*}
from \Eq{A1}.
Then,
\begin{align*}
T_1 u_1&=A_1 u_1-\frac1{N+1}(A_1 u_1,\bfone)\\
&=\frac1{N+3}\prescript{t\!}{}{\left(N+1,-1,-\frac{N}{N-1},\dots,-\frac{N}{N-1}\right)}\\
&=\frac1{N+3}\{(N+1)u_1-u_2\}\\
\shortintertext{and}
T_1 u_2&=A_1 u_2-\frac1{N+1}(A_1 u_2,\bfone)\\
&=A_1 u_2
=\frac1{N+3}u_2.
\end{align*}
The claim for $T_j$ is similarly proved.

(ii) Let $k\in\{i,j\}$. From \Lem{eigen}, $Z\subset E_k$. 
Therefore, for $u\in Z$, $A_k u=(N+3)^{-1}u\in \tilde l(V_0)$, which implies that $T_k u=A_k u= (N+3)^{-1}u$.
\end{proof}
In what follows, $L_i$ and $L_j$ also denote the linear operators $T_i|_Y$ and $T_j|_Y$, respectively.

Let us recall that a subset $C$ of a real vector space is called a \emph{cone} (with vertex at $0$) if the following hold:
\begin{itemize}
\item $C$ is convex.
\item For every $t\ge0$, $tC:=\{t u\mid u\in C\}$ is a subset of $C$.
\item $C\cap(-C)=\{0\}$.
\end{itemize}
A cone $C$ is called a \emph{proper cone} if $C$ is further closed and has nonempty interior.

We fix $M\,(>N)$ that is specified later.
We define a proper cone $C$ of $Y$ by
\[
C=\{s(M,-1)+t(1,-M)\mid s\ge0,\ t\ge0\}\subset \R^2\simeq Y.
\]
By introducing the projective coordinate $r=-\b/\a$ for $(\a,\b)\in\R^2$,
we see that $(\a,\b)\in C\setminus\{0\}$ if and only if $\a>0$ and $1/M\le r\le M$.
The maps $\tilde L_i$ and $\tilde L_j$ induced by $L_i$ and $L_j$ in the projective coordinate $r$ are expressed as
\[
\tilde L_i(r)=\frac{1+r}{N+1}
\quad\text{and}\quad
\tilde L_j(r)=\frac{(N+1)r}{1+r},
\]
respectively, from the expressions \Eq{expression_L}.
The unique fixed points of $\tilde L_i$ and $\tilde L_j$ in $(0,\infty)$ are $r=1/N$ and $r=N$, respectively. Hence, for any $M>N$, the interval $[1/M,M]$ contains both fixed points, and a direct computation shows that both maps send $[1/M,M]$ into its interior. That is, $L_k(C\setminus\{0\})\subset \mathring{C}$ for $k=i,j$, where $\mathring{C}$ denotes the interior of $C$.

Let $C^*\subset Y^*$ denote the dual cone of $C$ that is defined as $C^*=\{\ph\in Y^*\mid \ph(u)\ge0 \text{ for all }u\in C\}$. For $k=i,j$, the adjoint operator $L_k^*$ of $L_k$ satisfies $L^*_k(C^*\setminus \{0\})\subset \mathring{C^*}$, where $\mathring{C^*}$ is the interior of $C^*$ in $Y^*$.

Hilbert's projective metrics of $\mathring{C}$ and $\mathring{C^*}$ are denoted by $d_C$ and $d_{C^*}$, respectively. More precisely, 
\[
d_C(u,u')=\log\frac{\inf\{\lm>0\mid u\le_C \lm u'\}}{\sup\{\lm>0\mid \lm u'\le_C u\}},\quad u,u'\in\mathring{C},
\]
where $u\le_C v$ is defined by $v-u\in C$. $d_{C^*}$ is similarly defined.
It should be noted that $d_C(su,s'u')=d_C(u,u')$ for $s>0$ and $s'>0$, and $d_C(u,u')=0$ if and only if $u=su'$ for some $s>0$. The metric space $(\mathring{C}/{\sim},d_C)$ is complete, where $u\sim u'$ is defined as $d_C(u,u')=0$.

Since $C$ is a two-dimensional cone, we have the following expression in terms of the projective coordinate: when $1/M<r\le r'<M$,
\begin{equation}\label{eq:Hilbert_r}
  d_C(r,r')=\log\frac{(r'-1/M)(M-r)}{(r-1/M)(M-r')}.
\end{equation}
(see, e.g., \cite{Bi57,Nu88,EN95} for the general theory).
\begin{lem}\label{lem:diameter}
Let $\Delta=\max_{k=i,j}\diam_{d_C}L_k(C\setminus\{0\})$ and $\Delta^*=\max_{k=i,j}\diam_{d_{C^*}}L_k^*(C^*\setminus\{0\})$.
Then, $\Delta=\Delta^*=\log R_N$, where
\[
R_N=\frac{(M^2+M-N-1)\{(N+1)M^2-M-1\}}{M(M-N)(MN-1)}.
\]
\end{lem}
\begin{proof}
Since 
\[
\tilde L_i(1/M)=\frac{M+1}{M(N+1)}
\quad\text{and}\quad
\tilde L_i(M)=\frac{M+1}{N+1},
\]
the image of $[1/M,M]$ by $\tilde L_i$ is $[(M+1)/\{M(N+1)\},(M+1)/(N+1)]$.
By using \Eq{Hilbert_r},
\[
\diam_{d_C}L_i(C\setminus\{0\})=\log\frac{\left(\frac{M+1}{N+1}-\frac1M\right)\left\{M-\frac{M+1}{M(N+1)}\right\}}{\left\{\frac{M+1}{M(N+1)}-\frac1M\right\}\left(M-\frac{M+1}{N+1}\right)}
=\log R_N.
\]
By a similar calculation, we also have $\diam_{d_C}L_j(C\setminus\{0\})=\log R_N$.

From the general theory of Hilbert's projective metrics, the identity $\diam_{d_C}L_k(C\setminus\{0\})=\diam_{d_{C^*}}L^*_k(C^*\setminus\{0\})$ holds for $k=i,j$ (see, e.g., \cite[Lemma~17]{RKW11}).
This implies the claim.
\end{proof}
At this stage, we optimize $R_N$.
After a long but elementary calculation,
\[
\frac{d}{dM}(\log R_N)=\frac{(M^2-1)(N+1)(M^2-2NM+1)(NM^2-2M+N)}{M(M-N)(MN-1)(M^2+M-N-1)\{(N+1)M^2-M-1\}},
\]
which turns out that $R_N$ takes the minimal value $4N+5$ when $M=N+\sqrt{N^2-1}$.
In what follows, we set $M=N+\sqrt{N^2-1}$.

From the Birkhoff--Hopf theorem, it holds that
\begin{align}
d_C(L_k u,L_k u')&\le \tau d_C(u,u'),
\qquad k\in\{i,j\},\ u,u'\in\mathring{C},\label{eq:dC}\\
d_{C^*}(L^*_k \ph,L^*_k \ph')&\le \tau d_{C^*}(\ph,\ph'),
\qquad k\in\{i,j\},\ \ph,\ph'\in\mathring{C^*}\label{eq:dC*}
\end{align}
with
\[
\tau:=\tanh\frac{\Delta}{4}=\frac{\sqrt{R_N}-1}{\sqrt{R_N}+1}
=\frac{\sqrt{4N+5}-1}{\sqrt{4N+5}+1}=2^{-\gm_N}\in(0,1).
\]

We fix $\xi\in \mathring{C^*}$. Let $C_\xi$ be an affine section of $C$ that is defined as $C_\xi=\{y\in C\mid\xi(y)=1\}$.
Define
\[
\Sigma=\left\{\frac{L_k y}{\xi(L_k y)}\mathrel{}\middle|\mathrel{}y\in C_\xi,\ k\in\{i,j\}\right\}.
\]
Then, $\Sigma$ is a compact subset in $C_\xi\cap \mathring{C}$.
In the projective coordinate $r$, $\Sigma$ corresponds to a compact subset of the open interval $(1/M, M)$. 
Let $H(r)=\log\{(r-1/M)/(M-r)\}$ for $r\in(1/M,M)$. Then, $d_C(r,r')=|H(r)-H(r')|$ and $H'(r)=(r-1/M)^{-1}+(M-r)^{-1}$. Therefore, $H'$ is bounded above and below by positive constants on $\Sigma$, and Hilbert's projective metric and the Euclidean metric on $\Sigma$ are comparable: there exists $c>0$ such that
\begin{equation}\label{eq:comparable}
  c^{-1}d_C(y,y')\le |y-y'|\le c d_C(y,y'),\quad y,y'\in \Sigma.
\end{equation}
We set $q_1=(M,-1)$ and $q_2=(1,-M)\in\R^2\simeq Y$.
We let $q:=q_1+q_2\in \mathring{C}$ and define $C_q^*=\{\ph\in C^*\mid \ph(q)=1\}$. We also define
\[
\Sigma^*=\left\{\frac{L^*_k \ph}{(L^*_k\ph)(q)}\mathrel{}\middle|\mathrel{}\ph\in C_q^*,\ k\in\{i,j\}\right\}.
\]
Then, $\Sigma^*$ is a compact subset in $C^*_q\cap\mathring{C^*}$. By a similar argument as above, there exists (another) $c>0$ such that
\begin{equation}\label{eq:comparable*}
  c^{-1}d_{C^*}(\ph,\ph')\le |\ph-\ph'|\le c d_{C^*}(\ph,\ph'),\quad \ph,\ph'\in \Sigma^*.
\end{equation}

In what follows, positive constants $c$ may vary from line to line.

For $n\in\N$ and $w=w_1 w_2\cdots w_n\in\{i,j\}^n$, we define
\[
L_w=L_{w_n}L_{w_{n-1}}\cdots L_{w_1}
\text{ and }
L^*_w=L_{w_1}^*L_{w_2}^*\cdots L_{w_n}^*.
\]
\begin{lem}\label{lem:rankone}
There exists $c>0$ such that for any $n\in\N$ and $w\in\{i,j\}^n$,
\[
|\rho_w^{-1} L_w y - \ph_w(y)a_w|\le c \tau^n|y|,\quad y\in Y
\]
with $\rho_w=\xi(L_w q)=\xi(L_w q_1)+\xi(L_w q_2)>0$, $a_w=(L_w q_1)/\xi(L_w q_1)\in \Sigma$, and $\ph_w=(L_w^* \xi)/\rho_w\in\Sigma^*$.
Moreover, for $\om\in\{i,j\}^\N$, $\ph_{[\om]_n}$ converges to some $\ph_\om\in\Sigma^*$ and it holds that
\begin{equation}\label{eq:conv}
|\ph_{[\om]_n}-\ph_\om|\le c \tau^n,\quad n\in\N.
\end{equation}
Furthermore, if $\om,\om'\in\{i,j\}^\N$ satisfy $[\om]_m=[\om']_m$ for $m\in\N$,
\begin{equation}\label{eq:diff}
|\ph_\om-\ph_{\om'}|\le c\tau^m.
\end{equation}
\end{lem}
\begin{proof}
Let $n\in\N$ and $w\in\{i,j\}^n$.
For $k=1,2$, we define
\[
  s_k=\xi(L_w q_k)\quad\text{and}\quad a_k=\frac{L_w q_k}{s_k}\in\Sigma.
\]
From \Eq{comparable} and \Eq{dC},
\begin{align}
|a_1-a_2|&\le c d_C(a_1,a_2)
=c d_C(L_w q_1,L_w q_2)\label{eq:a1a2}\\
&\le c \tau^{n-1}\diam_{d_C}\Sigma\nonumber\\
&\le c\tau^n.\nonumber
\end{align}
We note that $a_w=a_1$ and $\ph_w(q_k)=s_k/(s_1+s_2)$ for $k=1,2$.

Any $y\in \R^2\simeq Y$ can be expressed as $y=\a_1 q_1+\a_2 q_2$ for some $\a_1,\a_2\in\R$.
Then,
\begin{align*}
&L_w y-\rho_w\ph_w(y)a_w\\
&=(\a_1 s_1 a_1+\a_2 s_2 a_2)-(s_1+s_2)\left(\a_1\cdot\frac{s_1}{s_1+s_2}+\a_2\cdot\frac{s_2}{s_1+s_2}\right)a_1\\
&=\a_2 s_2(a_2-a_1).
\end{align*}
Therefore,
\begin{align*}
|\rho_w^{-1}L_w y-\ph_w(y)a_w|
&=\left|\frac{\a_2 s_2}{s_1+s_2}(a_2-a_1)\right|\\
&\le|\a_2||a_1-a_2|\\
&\le c|y|\tau^n.\qquad\text{(from \Eq{a1a2})}
\end{align*}
Thus, the first claim follows.

For the second claim, let $w_n=[\om]_n$ for $n\in\N$. 
Since $L^*_{w_{n+1}}(C^*)\subset L^*_{w_{n}}(C^*)$ and $\diam_{d_{C^*}}L^*_{w_{n}}(C^*\setminus\{0\})$ converges to $0$ exponentially as $n\to\infty$, $\ph_{[\om]_n}$ converges to some $\ph_\om\in \Sigma^*$ exponentially fast in $d_{C^*}$. From \Eq{comparable*}, \Eq{conv} follows. This argument also shows that $\ph_\om$ is determined independently of the choice of $\xi$.

Lastly, suppose $\om,\om'\in\{i,j\}^\N$ have the same prefix $w\in\{i,j\}^m$. Then, by the definition of $\ph_\om$ in the above paragraph, both $\ph_\om$ and $\ph_{\om'}$ belong to $L^*_w(C^*\setminus\{0\})$.
Thus, 
\[
d_{C^*}(\ph_\om,\ph_{\om'})\le \diam_{d_{C^*}}L^*_w(C^*\setminus\{0\})
\le c \tau^{m-1}\diam_{d_{C^*}}\Sigma^*
\le c \tau^m.
\]
From \Eq{comparable*}, \Eq{diff} follows.
\end{proof}
We take $w\in W_*$ and $h\in\cH$. Let $u\in\tilde l(V_0)$ be defined as $u=P A_w \iota^{-1}(h)$.
$u$ is uniquely decomposed as $u=y+z$ with $y\in Y$ and $z\in Z$.
For each $k\in S$, let $u_k:=P A_w e_k=y_k+z_k$ with $y_k\in Y$ and $z_k\in Z$.
\begin{lem}\label{lem:decomposition}
Let $n\in\N$ and $w'\in\{i,j\}^n$. Then,
\begin{equation}\label{eq:decomposition}
\frac{\nu_h(K_{ww'})}{\nu(K_{ww'})}
=\frac{|L_{w'}y|^2+(N+3)^{-2n}|z|^2}{\sum_{k\in S}|L_{w'}y_k|^2+(N+3)^{-2n}\sum_{k\in S}|z_k|^2}.
\end{equation}
\end{lem}
\begin{proof}
From \Eq{cellenergy},
\begin{align*}
\nu_h(K_{ww'})&=2\left(\frac{N+3}{N+1}\right)^{|w|+n}\cE(\psi_{w'}^*\psi_w^* h,\psi_{w'}^*\psi_w^* h)\\
&=2\left(\frac{N+3}{N+1}\right)^{|w|+n}(T_{w'}u,-D T_{w'}u),
\end{align*}
where $T_{w'}=T_{w'_n}T_{w'_{n-1}}\cdots T_{w'_1}$ with $w'=w'_1w'_2\cdots w'_n$.
Since $-D|_{\tilde l(V_0)}=(N+1)I_{\tilde l(V_0)}$, $T_w'|_Y=L_{w'}$, and $T_{w'}|_Z=(N+3)^{-n}I_Z$ (from \Lem{invariant}),
\[
\nu_h(K_{ww'})=2(N+1)\left(\frac{N+3}{N+1}\right)^{|w|+n}\left\{|L_{w'}y|^2+(N+3)^{-2n}|z|^2\right\}.
\]
In the same way, we obtain that for each $k\in S$,
\[
\nu_{\iota(e_k)}(K_{ww'})=2(N+1)\left(\frac{N+3}{N+1}\right)^{|w|+n}\left\{|L_{w'}y_k|^2+(N+3)^{-2n}|z_k|^2\right\}.
\]
Thus, \Eq{decomposition} holds.
\end{proof}
We define a quadratic form $B_w\colon Y^*\times Y^*\to\R$ as
\[
B_w(\ph,\ph')=\sum_{k\in S}\ph(y_k)\ph'(y_k).
\]
\begin{lem}\label{lem:nondegeneracy}
$B_w$ is strictly positive-definite.
\end{lem}
\begin{proof}
Suppose that $B_w(\ph,\ph)=0$ for $\ph\in Y^*$.
Then, $\ph(y_k)=0$ for all $k\in S$.
We define $\hat\ph\in (l(V_0))^*$ so that $\hat\ph=\ph$ on $Y$, $\hat\ph=0$ on $Z$, and $\hat\ph\bfone=0$.
Then, for each $k\in S$,
\[
0=\ph(y_k)=\hat\ph(P A_w e_k)=\hat\ph(A_w e_k)=(A_w^* \hat\ph)(e_k).
\]
Therefore, $A_w^* \hat\ph=0$.
Since $A_w^*$ is invertible, $\hat\ph=0$. This implies that $\ph=0$.
\end{proof}
\begin{prop}\label{prop:limit}
Let $\om=\om_1\om_2\cdots\om_n\cdots\in\{i,j\}^\N$.
Then, the following identity holds.
\begin{equation}\label{eq:limit}
\dl_{h,w}(\om):=\lim_{n\to\infty}\frac{\nu_h(K_{w[\om]_n})}{\nu(K_{w[\om]_n})}=\frac{\ph_\om(y)^2}{B_w(\ph_\om,\ph_\om)},
\end{equation}
where $\ph_\om$ is provided in \Lem{rankone}.
Moreover, there exists $c>0$ independent of $\om$ such that
\[
  \sup_{\om\in\{i,j\}^\N}\left|\frac{\nu_h(K_{w[\om]_n})}{\nu(K_{w[\om]_n})}-\dl_{h,w}(\om)\right|\le c\tau^n,\quad n\in\N.
\]
\end{prop}
\begin{proof}
In the following, the constant in the error estimate can be chosen independently of $\om$.
From \Lem{rankone}, as $n\to\infty$,
\[
  L_{[\om]_n} y=\rho_{[\om]_n}\ph_{[\om]_n}(y)a_{[\om]_n}+O(\rho_{[\om]_n}\tau^n|y|).
\]
Since $\inf_n|a_{[\om]_n}|>0$,
\begin{align*}
|L_{[\om]_n} y|^2&=\rho_{[\om]_n}^2 |a_{[\om]_n}|^2 (\ph_{[\om]_n}(y)+O(\tau^n))^2\\
&=\rho_{[\om]_n}^2 |a_{[\om]_n}|^2 (\ph_{[\om]_n}(y)^2+O(\tau^n)).\end{align*}
In the same way, we have
\begin{equation}\label{eq:sum}
  \sum_{k\in S}|L_{[\om]_n} y_k|^2=\rho_{[\om]_n}^2 |a_{[\om]_n}|^2 (B_w(\ph_{[\om]_n},\ph_{[\om]_n})+O(\tau^n)).
\end{equation}
From \Lem{nondegeneracy}, $\inf_{\ph\in\Sigma^*}B_w(\ph,\ph)>0$. Therefore, the growth order of the right-hand side of \Eq{sum} is equal to that of $\rho_{[\om]_n}^2 |a_{[\om]_n}|^2$.

Now, we have
\begin{align*}
\rho_{[\om]_n}&=\xi(L_{[\om]_n}q_1)+\xi(L_{[\om]_n}q_2)\\
&\ge c (|L_{[\om]_n}q_1|+|L_{[\om]_n}q_2|)\\
&\ge c \|L_{[\om]_n}\|_{\mathrm{op}}\\
&\ge c\left|\det L_{[\om]_n}\right|^{1/2}\\
&=c\left\{\frac{N+1}{(N+3)^2}\right\}^{n/2}.
\end{align*}
Here, in the fourth line we used the inequality $\|X\|_{\mathrm{op}}^2\ge\left|\det X\right|$ for any matrices $X$ of size $2$. In the last line, we used the identity $\det L_k=(N+1)/(N+3)^2$ for $k\in\{i,j\}$.
Thus,
\[
\frac{(N+3)^{-2n}}{\rho_{[\om]_n}^2}=O((N+1)^{-n}),
\]
which is $O(\tau^n)$ because we can verify that $\tau\ge1/(N+1)$. 
From these estimates and \Lem{decomposition},
\begin{align*}
\frac{\nu_h(K_{w[\om]_n})}{\nu(K_{w[\om]_n})}
&=\frac{\rho_{[\om]_n}^2|a_{[\om]_n}|^2\left(\ph_{[\om]_n}(y)^2+O(\tau^n)\right)}{\rho_{[\om]_n}^2|a_{[\om]_n}|^2\left(\sum_{k\in S}\ph_{[\om]_n}(y_k)^2+O(\tau^n)\right)}\\
&=\frac{\ph_{[\om]_n}(y)^2+O(\tau^n)}{B_w(\ph_{[\om]_n},\ph_{[\om]_n})+O(\tau^n)}.
\end{align*}
Since $\inf_{\ph\in\Sigma^*}B_w(\ph,\ph)>0$ from \Lem{nondegeneracy},
\begin{align*}
\left|\frac{\nu_h(K_{w[\om]_n})}{\nu(K_{w[\om]_n})}-\frac{\ph_{[\om]_n}(y)^2}{B_w(\ph_{[\om]_n},\ph_{[\om]_n})}\right|
&=\frac{\left(\ph_{[\om]_n}(y)^2+B_w(\ph_{[\om]_n},\ph_{[\om]_n})\right)O(\tau^n)}{(B_w(\ph_{[\om]_n},\ph_{[\om]_n})+O(\tau^n))B_w(\ph_{[\om]_n},\ph_{[\om]_n})}\\
&=O(\tau^n).
\end{align*}
Since the function
\begin{equation}\label{eq:Fw}
F_w(\ph):=\frac{\ph(y)^2}{B_w(\ph,\ph)}, \quad \ph\in\Sigma^*
\end{equation}
is Lipschitz continuous on $\Sigma^*$,
\[
\left|\frac{\ph_{[\om]_n}(y)^2}{B_w(\ph_{[\om]_n},\ph_{[\om]_n})}-\frac{\ph_{\om}(y)^2}{B_w(\ph_{\om},\ph_{\om})}\right|
\le c|\ph_{[\om]_n}-\ph_{\om}|
\le c\tau^n,
\]
where the last inequality follows from \Eq{conv}.
By combining these estimates, the claim holds.
\end{proof}
\begin{lem}\label{lem:conti}
There exists $c>0$ such that the following holds: if $\om,\om'\in\{i,j\}^\N$ satisfy $[\om]_m=[\om']_m$ for some $m\in\Z_+$, then
\begin{equation}\label{eq:dlhom}
|\dl_{h,w}(\om)-\dl_{h,w}(\om')|\le c\tau^m.
\end{equation}
\end{lem}
\begin{proof}
The Lipschitz continuity of $F_w$ defined in \Eq{Fw} on $\Sigma^*$, \Eq{limit}, and \Eq{diff} imply that \Eq{dlhom} follows for $m\in\N$ for some suitable $c>0$.
From the boundedness of $\dl_{h,w}$, \Eq{dlhom} holds also for $m=0$ by choosing $c$ larger.
\end{proof}
Recall the map $\Psi_w$ defined before the statement of \Thm{main3}.
\begin{lem}\label{lem:coincidence}
If $\Psi_w(\om)=\Psi_w(\om')$ for $\om,\om'\in\{i,j\}^\N$, then $\dl_{h,w}(\om)=\dl_{h,w}(\om')$.
\end{lem}
\begin{proof}
It suffices to consider the case that
$\om=\eta ij^\infty$ and $\om'=\eta ji^\infty$ for some $\eta\in \bigcup_{n\ge0}\{i,j\}^n$.
Let $u=\iota^{-1}(h)\in l(V_0)$. From \Lem{3},
\begin{equation}\label{eq:2.2-1}
\dl_{h,w}(\om)
=\lim_{n\to\infty}\frac{\nu_h(K_{w\eta i j^n})}{\nu(K_{w\eta i j^n})}
=\frac{(d_j,A_{w\eta i}u)^2}{\sum_{k\in S}(d_j,A_{w\eta i} e_k)^2}=\frac{({}^t\!A_id_j,A_{w\eta}u)^2}{\sum_{k\in S}({}^t\!A_id_j,A_{w\eta} e_k)^2}
\end{equation}
and
\begin{equation}\label{eq:2.2-2}
\dl_{h,w}(\om')
=\lim_{n\to\infty}\frac{\nu_h(K_{w\eta j i^n})}{\nu(K_{w\eta j i^n})}
=\frac{({}^t\!A_jd_i,A_{w\eta}u)^2}{\sum_{k\in S}({}^t\!A_jd_i,A_{w\eta} e_k)^2}.
\end{equation}
Since a direct computation shows that ${}^t\!A_id_j=-{}^t\!A_jd_i$, the rightmost expressions of \Eq{2.2-1} and \Eq{2.2-2} coincide.
\end{proof}
From this lemma, the function $\frac{\dl \nu_h}{\dl \nu}(x):=\dl_{h,w}(\om)$ for $x=\Psi_w(\om)$, $x\in\psi_w(\overline{p_ip_j})$ is well-defined.
\begin{proof}[Proof of \Thm{main3}.]
It remains to prove the $\gm_N$-H\"older continuity of $\frac{\dl \nu_h}{\dl \nu}$.
Take any distinct $x,x'\in \psi_w(\overline{p_i p_j})$.
There exist a unique integer $m$ such that
\begin{equation}\label{eq:compare}
  2^{-(m+1)}<2^{|w|}\frac{|x-x'|}{|p_i-p_j|}\le 2^{-m}.
\end{equation}
We can take $x_0=x$, $x_1$, $x_2=x'$ so that $x_1$ is one of the points dividing the line segment $\psi_w(\overline{p_i p_j})$ into $2^m$ equal parts, and $|x_{l+1}-x_l|\le 2^{-m-|w|}|p_i-p_j|$ for $l=0,1$.
For each $l=0,1$, we can take $\om,\om'\in\{i,j\}^\N$ such that $x_l=\Psi_w(\om)$, $x_{l+1}=\Psi_w(\om')$ and $[\om]_m=[\om']_m$.
From \Lem{conti},
\[
  \left|\frac{\dl \nu_h}{\dl \nu}(x_{l+1})-\frac{\dl \nu_h}{\dl \nu}(x_l)\right|\le c\tau^m=c 2^{-\gm_N m}.
\]
Therefore,
\[
  \left|\frac{\dl \nu_h}{\dl \nu}(x)-\frac{\dl \nu_h}{\dl \nu}(x')\right|
  \le \sum_{l=0}^1 \left|\frac{\dl \nu_h}{\dl \nu}(x_{l+1})-\frac{\dl \nu_h}{\dl \nu}(x_l)\right|\le 2c 2^{-\gm_N m}
    \le c'|x-x'|^{\gm_N}
\]
with another constant $c'>0$.
Here, the last inequality follows from \Eq{compare}.
This proves the claim.
\end{proof}
\begin{rem}
Although $\dl\nu_h/\dl\nu$ is defined edge by edge, its value at a point does not depend on the edge containing it. Indeed, let $x$ lie on two distinct edges $E$ and $E'$. If one of them contains the other, say $E'=\psi_{wv}(\overline{p_ip_j})\subset E=\psi_w(\overline{p_ip_j})$ with $v\in\{i,j\}^m$, then $\dl_{h,wv}(\om)=\dl_{h,w}(v\om)$ for $\om\in\{i,j\}^\N$ directly from \Eq{dlh}. Otherwise $E\cap E'=\{x\}$ and $x\in V_*$. If $x=p_k\in V_0$, the only cells containing $x$ are $K_{k^n}$, $n\ge0$, and the claim is trivial. If $x\in V_*\setminus V_0$, write $x=\psi_u(p_{kl})=\psi_{uk}(p_l)=\psi_{ul}(p_k)$ with $u\in W_*$ and distinct $k,l\in S$; then the cells of level at least $|u|+1$ containing $x$ are exactly $K_{ukl^r}$ and $K_{ulk^r}$, $r\ge0$. Consequently, along any edge through $x$, the cells shrinking to $x$ eventually form one of the two sequences $\{K_{ukl^n}\}_n$ and $\{K_{ulk^n}\}_n$, and the corresponding limits in \Eq{dlh} coincide by \Lem{coincidence}. Therefore $\dl\nu_h/\dl\nu$ is a well-defined function on the union of all edges of $K$.
\end{rem}
\section{Proof of Theorem~\protect\ref{th:main4}}\label{sec:main4}
Let $h\in\cH$, $w\in W_*$, and $i,j\in S$ be distinct elements.
We investigate the behavior of $\dl_{h,w}$ defined in \Eq{limit} around $w v \chi$ for $v\in\{i,j\}^m$ with some $m\in\Z_+$ and $\chi:=(ij)^\infty\in\{i,j\}^\N$. Recall \Lem{rankone} and its proof. $\dl_{h,w}(\om)$ is identified by using $\ph_\om$, and $\ph_\om$ is given by the limit of the normalization of $L^*_{[\om]_n}\xi$ for a fixed $\xi\in \mathring{C^*}$. Therefore, what are important are the matrices $L_i$ and $L_j$ in \Eq{expression_L} modulo multiplicative constants. Keeping this observation in mind, we define matrices
\[
\bar L_i=\begin{pmatrix} N+1 & 0\\-1& 1\end{pmatrix}
\quad\text{and}\quad
\bar L_j=\begin{pmatrix} 1 & -1\\0& N+1\end{pmatrix}
\]
and
\[
  B:=\bar L_j \bar L_i=\begin{pmatrix} N+2 & -1\\-(N+1)& N+1\end{pmatrix}.
\]
The eigenvalues $\lm_\pm$ of $B$ are given by
\[
  \lm_\pm=\frac{2N+3\pm\sqrt{4N+5}}{2}.
\]
Let $\theta=\lm_-/\lm_+$. A direct calculation shows that
\begin{equation}\label{eq:theta}
\sqrt{\theta}=2^{-\gm_N}.
\end{equation}
Thus $\theta=\tau^2$, where $\tau$ is the contraction ratio appearing in \Sec{main3}.

The eigenvectors of ${}^t\!B$ corresponding to $\lm_+$ and $\lm_-$ are given by
\[
\ell_+:=\begin{pmatrix}\sqrt{4N+5}+1\\-2\end{pmatrix}\quad\text{and}\quad
\ell_-:=\begin{pmatrix}\sqrt{4N+5}-1\\2\end{pmatrix},
\]
respectively.
We regard $\ell_\pm$ as elements of $Y^*$ by the identification $Y^*\simeq(\R^2)^*\simeq\R^2$.
From the construction of $\ph_\chi$ (\Lem{rankone}), $\ph_\chi$ coincides with $\ell_+$ up to a multiplicative constant.

Let $u=PA_w\iota^{-1}(h)\in\tilde l(V_0)$ and write $u=y+z$ with $y\in Y$ and $z\in Z$.
Recall the function $F_w$ defined in \Eq{Fw}.
We will also regard $F_w$ as a smooth homogeneous function on $Y^*\setminus\{0\}$.
This satisfies the relation $F_w(t\ph)=F_w(\ph)$ for all $t\in\R\setminus\{0\}$ and $\ph\in Y^*\setminus\{0\}$.
In particular, we can consider the differentiation $\partial_{\eta}F_w$ of $F_w$ along $\eta\in Y^*$.

Let $m\in\Z_+$ and $v=v_1v_2\cdots v_m\in\{i,j\}^m$.  We set
\[
    \bar L_v=\bar L_{v_m}\bar L_{v_{m-1}}\cdots \bar L_{v_1},
    \qquad
    \ell_+^v={}^t\!\bar L_v\ell_+,
    \qquad
    \ell_-^v={}^t\!\bar L_v\ell_-.
\]
For $v=\emptyset$, we understand that
$\ell_+^\emptyset=\ell_+$ and $\ell_-^\emptyset=\ell_-$.
\begin{lem}\label{lem:optimality}
Suppose that 
\begin{equation}\label{eq:nondeg}
\partial_{\ell_-^v}F_w(\ell_+^v)\ne0.
\end{equation}
Let $x=\Psi_w(v\chi)\in\psi_w(\overline{p_i p_j})$.
Then, there exist a constant $c>0$ and a sequence $\{x_n\}_{n=1}^\infty$ of $\psi_w(\overline{p_i p_j})\setminus\{x\}$ converging to $x$ such that 
\begin{equation}\label{eq:optimality}
  \left|\frac{\dl\nu_h}{\dl\nu}(x_n)-\frac{\dl\nu_h}{\dl\nu}(x)\right|
  \ge c|x_n-x|^{\gm_N}, \quad n\in\N.
\end{equation}
In particular, \Eq{reverse} holds.
\end{lem}
\begin{proof}
Let $\check\chi=(ji)^\infty\in\{i,j\}^\N$.
We define
\[
  \check B:=\bar L_i \bar L_j=\begin{pmatrix} N+1 & -(N+1)\\ -1&N+2\end{pmatrix}.
\]
The eigenvalues of $^t\!\check B$ are $\lm_\pm$, and the eigenvector corresponding to $\lm_+$ is $\check\ell_+:=\begin{pmatrix}-2\\\sqrt{4N+5}+1\end{pmatrix}$.
From the construction of $\ph_{\check\chi}\in Y^*\simeq\R^2$, it coincides with $\check\ell_+$ up to a multiplicative constant.
It holds that $\check\ell_+=a\ell_++b\ell_-$ with
\[
a=-\frac{N+2}{\sqrt{4N+5}},\quad b=\frac12+\frac{2N+1}{2\sqrt{4N+5}}.
\]
In particular, $a\ne0$ and $b\ne0$.

For $n\in\N$, let $\om_n=v\underbrace{ijij\cdots ij}_{2n}\check\chi\in\{i,j\}^\N$ and $x_n=\Psi_w(\om_n)$.
Then, $\ph_{\om_n}$ is parallel to $^t\!\bar L_v(^t\!B)^n  \check\ell_+=a\lm_+^n\ell_+^v+b\lm_-^n\ell_-^v$.
Since $F_w$ is homogeneous of degree zero,
\[
  F_w(\ph_{v\chi})=F_w(\ell_+^v)
  \quad\text{and}\quad 
  F_w(\ph_{\om_n})=F_w\left(\ell_+^v+\frac{b}{a}\theta^n\ell_-^v\right).
\]
Therefore,
\begin{align*}
\frac{\dl \nu_h}{\dl \nu}(x_n)-\frac{\dl \nu_h}{\dl \nu}(x)
&=\dl_{h,w}(\om_n)-\dl_{h,w}(v\chi)\\
&=F_w(\ph_{\om_n})-F_w(\ph_{v\chi})\\
&=F_w\left(\ell_+^v+\frac{b}{a}\theta^n\ell_-^v\right)-F_w(\ell_+^v)\\
&=\frac{b}{a}\theta^n\partial_{\ell_-^v}F_w(\ell_+^v)+O(\theta^{2n})
\quad\text{(from Taylor's theorem)}\\
&=\frac{b}{a}2^{-2\gm_N n}\partial_{\ell_-^v}F_w(\ell_+^v)+O(2^{-4\gm_N n}).\quad\text{(from \Eq{theta})}
\end{align*}
From the assumption \Eq{nondeg}, there exists $c'>0$ such that for sufficiently large $n$,
\begin{equation}\label{eq:ge}
  \left|\frac{\dl \nu_h}{\dl \nu}(x_n)-\frac{\dl \nu_h}{\dl \nu}(x)\right|\ge c'2^{-2\gm_N n}.
\end{equation}
On the other hand, $x_n$ and $x$ are two distinct points that divide the line segment $\psi_{\text{\scriptsize $w v \underbrace{ijij\cdots ij}_{2n}$}}(\overline{p_i p_j})$ into three parts. 
Therefore,
\begin{equation}\label{eq:xnx}
|x_n-x|=\frac13 2^{-|w|-m-2n}|p_i-p_j|.
\end{equation}
Combining \Eq{ge} and \Eq{xnx}, we obtain \Eq{optimality} by relabeling the indices of the sequence $\{x_n\}$ if necessary.
\end{proof}
\begin{lem}\label{lem:critical}
Let $y\in Y\setminus\{0\}$.  Regard the function $F_w$ defined in \Eq{Fw} as a function on the projective line $\mathbb P(Y^*)$.  Then $F_w$ has exactly two critical points on $\mathbb P(Y^*)$.
\end{lem}
\begin{proof}
Since $B_w$ is strictly positive definite, there exists a unique
$\xi\in Y^*$ such that
\[
    \varphi(y)=B_w(\varphi,\xi),
    \qquad \varphi\in Y^* .
\]
Since $y\neq0$, we have $\xi\neq0$.  Choose a $B_w$-orthonormal basis $\{\xi_1,\xi_2\}$ of $Y^*$ with $\xi_1=B_w(\xi,\xi)^{-1/2}\xi$.  Writing $\varphi=(r\cos t)\xi_1+(r\sin t)\xi_2$, we obtain
\[
    F_w(\varphi)=\frac{B_w(\ph,\xi)^2}{B_w(\ph,\ph)}=B_w(\xi,\xi)\cos^2 t.
\]
Hence, as a function on the projective line, $F_w$ has precisely two
critical points: the line spanned by $\xi$, and the $B_w$-orthogonal
line to $\xi$.  This proves the claim.
\end{proof}
\begin{lem}\label{lem:012}
Let $y\in Y\setminus\{0\}$.  Then there exists
$m\in\{0,1,2\}$ such that \Eq{nondeg} holds with $v=i^m$.\end{lem}

\begin{proof}
By \Lem{critical}, $F_w$ has only two critical points on
$\mathbb P(Y^*)$.
We can confirm
\[
    ({}^t\!\bar L_i)^m\ell_+
    =
    \begin{pmatrix}
        (N+1)^m(\sqrt{4N+5}+1)+\dfrac{2}{N}\{(N+1)^m-1\}\smallskip\\
        -2
    \end{pmatrix}
\]
from the concrete expressions of $\bar L_i$ and $\ell_+$ and by induction.
The first components are mutually distinct for $m=0,1,2$, while the
second component is always $-2$.  Hence these three vectors determine
three distinct points of $\mathbb P(Y^*)$.
Therefore at least one of them is not a critical point of $F_w$.
Choose such an $m\in\{0,1,2\}$, and put $v=i^m$.  Then $\ell_+^v=({}^t\!\bar L_i)^m\ell_+$ is not a critical point of $F_w$ on $\mathbb P(Y^*)$.  
Since ${}^t\!\bar L_v$ is
invertible and $\ell_+$ and $\ell_-$ are linearly independent,
$\ell_+^v$ and $\ell_-^v$ are linearly independent.  Thus
$\ell_-^v$ gives a nonzero tangent direction at
$[\ell_+^v]\in\mathbb P(Y^*)$.  Since $[\ell_+^v]$ is not a
critical point of $F_w$, we obtain \Eq{nondeg}.
\end{proof}
\begin{proof}[Proof of \Thm{main4}.]
(i) Let $u=P A_w\iota^{-1}(h)\in \tilde l(V_0)$.
Since $h$ is not constant and $A_w$ is invertible on
$l(V_0)/\mathbb R{\bf 1}$, we have $u\neq0$.  Hence there exist
distinct $i,j\in S$ such that the $Y$-component $y=(u(p_i),u(p_j))\in Y$ is nonzero.
By \Lem{012}, there exists $m\in\{0,1,2\}$ such that \Eq{nondeg} holds with $v=i^m$.
Applying Lemma 7.1, we obtain a point $x=\Psi_w\bigl(v\chi\bigr)\in\psi_w(\overline{p_ip_j})$ at which the estimate in \Thm{main4}~(i) holds.

(ii) Fix distinct $i,j\in S$.
For each $y\in Y\setminus\{0\}$, let $\xi_y\in Y^*$ be the unique element
such that
$\varphi(y)=B_w(\varphi,\xi_y)$ for $ \varphi\in Y^*$.
By the proof of \Lem{critical}, the point $[\ell_+]$ is critical for
$F_w$ if and only if $\xi_y$ is parallel to $\ell_+$ or
$B_w$-orthogonal to $\ell_+$. Since the map $y\mapsto\xi_y$ is
an isomorphism, the set of such $y$'s is the union of two proper
linear subspaces of $Y$. Hence we may choose $y\in Y\setminus\{0\}$
so that $[\ell_+]$ is not a critical point of $F_w$.
Take $\tilde u\in\tilde l(V_0)$ whose $Y$-component equals
$y$.  Since the linear operator $P A_w$ on $\tilde l(V_0)$ is an isomorphism, there exists $u\in\tilde l(V_0)$ such that
$P A_wu=\tilde u$.  Put $h=\iota(u)$.  Then the corresponding
$Y$-component of $P A_w\iota^{-1}(h)$ is $y$, and hence
$\partial_{\ell_-}F_w(\ell_+)\ne0$.
Applying \Lem{optimality} with $v=\emptyset$, we obtain the desired conclusion.
\end{proof}
\section*{Use of AI tools}
The authors used AI during the preparation of this manuscript for obtaining comments on exposition and possible proof strategies. All mathematical arguments, statements, and references were independently checked and verified by the authors, who take full responsibility for the final manuscript.

\end{document}